\pdfoutput=1
\RequirePackage{ifpdf}
\ifpdf % We are running pdfTeX in pdf mode
\documentclass[pdftex]{sigma}
\else
\documentclass{sigma}
\fi

\usepackage{tikz}
\usepackage{tikz-cd}
 \usetikzlibrary{shapes}

\usepackage[shortlabels]{enumitem}
\usepackage{bm}
\usepackage{longtable}
\usepackage{float}

\DeclareMathOperator{\im}{im}

\DeclareMathOperator{\GL}{GL}
\DeclareMathOperator{\SL}{SL}
\DeclareMathOperator{\PSL}{PSL}

\DeclareMathOperator{\tr}{Tr}

\DeclareMathOperator{\ord}{ord}

\DeclareMathOperator{\Irrep}{Irrep}
\DeclareMathOperator{\AL}{AL}

\newcommand{\lmod}{\backslash}

\newcommand{\N}{\mathbb{N}}
\newcommand{\Z}{\mathbb{Z}}
\newcommand{\Q}{\mathbb{Q}}
\newcommand{\R}{\mathbb{R}}

\renewcommand{\H}{\mathbb{H}}
\newcommand{\F}{\mathbb{F}}
\newcommand{\M}{\mathbb{M}}
\newcommand{\mf}{\mathfrak}

\newcommand{\Zmod}[1]{\mathbb{Z}/ #1\mathbb{Z}}

\newcommand{\Matrix}[1]{\begin{pmatrix} #1 \end{pmatrix}}
\newcommand{\Mod}[1]{\ \left({\rm mod}\ #1\right)}

\newcommand\prn[1]{\left(#1\right)}

\newcommand\gen[1]{\left\langle#1\right\rangle}

\newcommand{\qr}[2]{\left(\frac{#1}{#2}\right)}
\newcommand{\al}[2]{#1\in\operatorname{AL}(#2)}

\newcommand{\nth}[1]{#1^\text{th}}

\newcommand{\conj}[1]{\overline{#1}}

\newcommand{\psring}[2]{#1[\![#2]\!]}
\newcommand{\divides}{\mid}
\newcommand{\exact}{\|}

\renewcommand{\a}{\alpha}

\newcommand{\g}{\gamma}
\newcommand{\G}{\Gamma}

\newcommand{\e}{\epsilon}

\renewcommand{\t}{\tau}

\newcommand{\mt}{\mathcal T}

\newcommand{\inftyzero}{
 \tikzset{
 infty/.style={rectangle, draw=black},
 zero/.style={fill=lightgray}
 }
}

\makeatletter
\let\c@table\c@figure
\makeatother

\numberwithin{equation}{section}
\numberwithin{table}{section}
\numberwithin{figure}{section}

\newtheorem{Theorem}{Theorem}[section]
\newtheorem{Corollary}[Theorem]{Corollary}
\newtheorem{Lemma}[Theorem]{Lemma}
\newtheorem{Proposition}[Theorem]{Proposition}
\newtheorem{Conjecture}[Theorem]{Conjecture}
 { \theoremstyle{definition}
\newtheorem{Definition}[Theorem]{Definition}
\newtheorem{Example}[Theorem]{Example}
\newtheorem{Remark}[Theorem]{Remark} }

\begin{document}

\allowdisplaybreaks

\newcommand{\arXivNumber}{1809.02913}

\renewcommand{\thefootnote}{}

\renewcommand{\PaperNumber}{033}

\FirstPageHeading

\ShortArticleName{$p$-Adic Properties of Hauptmoduln with Applications to Moonshine}

\ArticleName{$\boldsymbol{p}$-Adic Properties of Hauptmoduln\\ with Applications to Moonshine\footnote{This paper is a~contribution to the Special Issue on Moonshine and String Theory. The full collection is available at \href{https://www.emis.de/journals/SIGMA/moonshine.html}{https://www.emis.de/journals/SIGMA/moonshine.html}}}

\Author{Ryan C.~CHEN, Samuel MARKS and Matthew TYLER}

\AuthorNameForHeading{R.C.~Chen, S.~Marks and M.~Tyler}

\Address{Department of Mathematics, Princeton University, Princeton, NJ 08544, USA}
\Email{\href{mailto:rcchen@princeton.edu}{rcchen@princeton.edu},
 \href{mailto:spmarks@princeton.edu}{spmarks@princeton.edu},
 \href{mailto:mttyler@princeton.edu}{mttyler@princeton.edu}}

\ArticleDates{Received September 19, 2018, in final form April 10, 2019; Published online April 29, 2019}

\Abstract{The theory of monstrous moonshine asserts that the coefficients of Hauptmo\-duln, including the $j$-function, coincide precisely with the graded characters of the monster module, an infinite-dimensional graded representation of the monster group. On the other hand, Lehner and Atkin proved that the coefficients of the $j$-function satisfy congruences modulo $p^n$ for $p \in \{2, 3, 5, 7, 11\}$, which led to the theory of $p$-adic modular forms. We combine these two aspects of the $j$-function to give a general theory of congruences modulo powers of primes satisfied by the Hauptmoduln appearing in monstrous moonshine. We prove that many of these Hauptmoduln satisfy such congruences, and we exhibit a relationship between these congruences and the group structure of the monster. We also find a distinguished class of subgroups of the monster with graded characters satisfying such congruences.}


\Classification{11F11; 11F22; 11F33}

\renewcommand{\thefootnote}{\arabic{footnote}}
\setcounter{footnote}{0}

\section{Introduction and statements of results}\label{sec:intro}
The theory of monstrous moonshine arose from the remarkable observation of McKay and Thompson \cite{Thompson79} that
\begin{gather*}196884 = 1 + 196883\end{gather*}
and its generalizations, including
\begin{gather*}
 21493760 = 1 + 196883 + 21296876, \\
 864299970 = 2 \times 1 + 2 \times 196883 + 21296876 + 842609326, \\
 20245856256 = 2 \times 1 + 3 \times 196883 + 2 \times 21296876 + 842609326 + 19360062527.
\end{gather*}
Here, the left-hand sides of the equations are the coefficients of the \textit{normalized modular $j$-function}
\begin{gather*}J(\tau) = j(\tau) - 744 = q^{-1} + 196884q + 21493760q^2 + \cdots,\qquad\text{where} \quad q = {\rm e}^{2\pi {\rm i} \tau},\end{gather*}
and the right-hand sides are simple sums involving the dimensions of the irreducible representations of the monster group $\M$:
\begin{gather*}1,\; 196883,\; 21296876,\; 842609326,\; 19360062527,\; \dots.\end{gather*}
Thompson conjectured \cite{Thompson79b} that these identities could be explained by the existence of an infinite-dimensional graded \textit{monster module}
\begin{gather*}V^\natural = \bigoplus_{n=-1}^\infty V^\natural_n\end{gather*}
such that the graded dimension is given by $J$. More generally, the graded-trace functions
\begin{gather*}\mathcal{T}_g(\t) = \sum_{n=-1}^\infty \tr\big(g | V_n^\natural\big) q^n\end{gather*}
for the action of $\M$ on $V^\natural$ are known as the \textit{McKay--Thompson series} and depend only on the conjugacy class of $g\in \M$. As part of their famous monstrous moonshine conjectures, Conway and Norton computed for each monster conjugacy class $g$ a genus zero group $\Gamma_g\le \GL_2^+(\R)$ on which they conjectured $\mathcal T_g$ was a \textit{normalized Hauptmodul}~\cite{Conway79}. That is, each $\mathcal{T}_g$ was conjectured to be the unique generator $\mathcal{T}_{\G_g}$ of the function field of the genus zero curve $\Gamma_g \lmod \H^*$ having $q$-expansion of the form $q^{-1} + O(q)$ at $\infty$. Since all of the Hautpmoduln appearing in this paper will be normalized (meaning that they are bounded away from $\infty$ and have $q$-expansion $q^{-1} + O(q)$ at $\infty$), we will henceforth omit the word ``normalized'' and refer to such functions simply as \textit{Hautpmoduln}. Frenkel--Lepowsky--Meurman \cite{Frenkel84,Frenkel85} constructed $V^\natural$ with the correct graded dimensions, and Borcherds \cite{Borcherds92} proved that the McKay--Thompson series were Hauptmoduln for the $\Gamma_g$ given by Conway--Norton. After the proof of monstrous moonshine, different incarnations of moonshine were shown for other finite groups, such as the largest Mathieu group~$M_{24}$~\cite{Gannon16}, and later the other $22$ groups appearing in umbral moonshine \cite{Duncan15}. There is also a notion of generalized moonshine, conjectured by Norton~\cite{Norton87} and recently proved by Carnahan~\cite{Carnahan16}.

Thirty years before the observation of McKay and Thompson, Lehner \cite{Lehner49a,Lehner49b} and Atkin \cite{Atkin67} proved that the Fourier expansion of $J(\t) = q^{-1} + \sum c(n) q^n$ satisfies the following congruences for all positive $\a$:
\begin{gather}
 c(2^\alpha n) \equiv 0\Mod{2^{3\alpha + 8}}, \notag\\
 c(3^\alpha n) \equiv 0\Mod{3^{2\alpha + 3}}, \notag\\
 c(5^\alpha n) \equiv 0\Mod{5^{\alpha + 1}},\notag\\
 c(7^\alpha n) \equiv 0\Mod{7^\alpha}, \notag \\
 c(11^\alpha n) \equiv 0\Mod{11^\alpha}.\label{eqn:j_congruences}
\end{gather}
Viewed another way, these identities state that $J|U_p^n$ uniformly converges to zero $p$-adically as $n \to \infty$, where $U_p$ is the operator defined on $q$-expansions by
\begin{gather*}
 \prn{\sum a(n)q^n}|U_p = \sum a(pn)q^n.
\end{gather*}
Such congruences led Serre, Katz, and others to develop a robust and fruitful theory of $p$-adic modular forms \cite{Calegari,Gouvea,Katz73,Katz76,Serre72}.

Given the deep connections between $J$ and the monster, one might wonder whether these $p$-adic properties of $J$ are special cases of a more general $p$-adic phenomenon taking place among the Hauptmoduln appearing in monstrous moonshine. To make this more precise, given a~prime~$p$ and a~modular function~$f$, we say that~$f$ is \textit{$p$-adically annihilated} if the $q$-series $f|U_p^n$ uniformly converges to $0$ in the $p$-adic limit as $n\rightarrow\infty$. Given that $J$ is $p$-adically annihilated for $p\in\{2,3,5,7,11\}$, we can then ask if other Hauptmoduln appearing in monstrous moonshine are as well.

There is some literature studying coefficient congruences of a related nature. The papers \cite{Andersen13, Jenkins15} discuss Hauptmoduln on $\Gamma_0(N)$, and \cite{Thompson79b} discusses other coefficient congruences involving Hauptmoduln. However, there has not been a systematic study of $p$-adic annihilation for all of the monstrous moonshine Hauptmoduln.

Our first main result is that $p$-adic annihilation is actually quite common among the Hauptmoduln of monstrous moonshine. In fact, out of the $171$ Hauptmoduln in monstrous moonshine, we will show that $97$ have $p$-adic annihilation for some prime $p$.
\begin{Theorem}\label{thm:annihilation} For primes $p\in\{2,3,5,7,11\}$, the Hauptmoduln $\mt$ corresponding to the genus zero groups of monstrous moonshine appearing in Table~{\rm \ref{tbl:Up_annihilation}} are $p$-adically annihilated.
\end{Theorem}
We further conjecture that Table~\ref{tbl:Up_annihilation} gives \textit{all} the Hauptmoduln appearing in monstrous moonshine that are $p$-adically annihilated for any prime $p$ (see Conjecture~\ref{conj:annihilation_converse}).

Once Theorem~\ref{thm:annihilation} has established a class of Hauptmoduln coming from monstrous moonshine with $p$-adic annihilation, we may next ask whether the structure of the monster group informs $p$-adic properties of the Hauptmoduln. Specifically, we are interested in relating the \textit{power maps} $g\mapsto g^m$ of the monster (or equivalently, the corresponding maps of conjugacy classes) to $p$-adic annihilation of Hauptmoduln.
\begin{Theorem}\label{thm:power} Let $\mt_g$ be the Hauptmodul of a group appearing in Table~{\rm \ref{tbl:Up_annihilation}}, so that $\mt_g$ is $p$-adically annihilated by Theorem~{\rm \ref{thm:annihilation}}. Outside of the exceptions discussed in Section~{\rm \ref{ssec:A-power}}, we also have that $\mt_{g^m}$ is $p$-adically annihilated for any $m \in \N$.
\end{Theorem}
Although Theorem~\ref{thm:power} follows from Theorem~\ref{thm:annihilation}, we will prove the two theorems in tandem, relying on the structure provided by Theorem~\ref{thm:power} to make Theorem~\ref{thm:annihilation} easier to prove. As an illustration of Theorem~\ref{thm:power}, see Fig.~\ref{fig:2web}, which shows conjugacy classes with Hauptmoduln that are $2$-adically annihilated and the power maps between them. For a full explanation of the notation used in this figure, and the corresponding figures for $p = 3,5,7,11$, see Appendix~\ref{apndx:web}.
\begin{figure}[t] \centering
 \begin{tikzpicture}[scale=2]
 \inftyzero
 \small{
 \node at ( 0, 0) (1) {$1$};
 \node at (-.5, -1) (2) {$2$};
 \node at ( 0, -1) (2+) {$2{+}$};
 \node at ( 1, -1) (3I3) {$3|3$};
 \node at ( 2, -1.5) (3+) {$3{+}$};
 \node[infty] at (-1, -2) (4) {$4$};
 \node at (-.5, -2) (4+) {$4{+}$};
 \node at ( 1, -2) (6I3) {$6|3$};
 \node at ( 3, -2) (5+) {$5{+}$};
 \node at (1.75, -2.5) (6+3) {$6{+}3$};
 \node at (2.25, -2.5) (6+) {$6{+}$};
 \node at ( 4, -2.5) (11+) {$11{+}$};
 \node[infty] at (-1.5, -3) (8) {$8$};
 \node at (-1, -3) (8+) {$8{+}$};
 \node at ( 1, -3) (12I3+) {$12|3{+}$};
 \node at (2.75, -3) (10+5) {$10{+}5$};
 \node at (3.25, -3) (10+) {$10{+}$};
 \node at ( 2, -3.5) (12+) {$12{+}$};
 \node at (3.75, -3.5) (22+11) {$22{+}11$};
 \node at (4.25, -3.5) (22+) {$22{+}$};
 \node[infty] at (-2, -4) (16) {$16$};
 \node at (-1.5, -4) (16+) {$16{+}$};
 \node[infty] at ( 0, -4) (12+3) {$12{+}3$};
 \node at ( 3, -4) (20+) {$20{+}$};
 \node at (-.5, -5) (24+) {$24{+}$};
 \node at ( 4, -4.5) (44+) {$44{+}$};
 \node at (-2, -5) (32+) {$32{+}$};

 \draw (1) -- (2) --(4) -- (8) -- (16) -- (32+);
 \draw (16+) -- (8) -- (24+) -- (12+3) -- (4) -- (8+);
 \draw (1) -- (3I3) -- (6I3) -- (12I3+) -- (4+) -- (2) -- (6I3);
 \draw (1) -- (2+);
 \draw (1) to[out=-30, in=120] (3+);
 \draw (1) to[out=-20, in=120] (5+);
 \draw (1) to[out=-10, in=120] (11+);
 \draw (2) to[out=-40, in=165] (6+3);
 \draw (2) to[out=-50, in=165] (10+5);
 \draw (2) to[out=-60, in=165] (22+11);
 \draw (2+) to[out=-20] (22+);
 \draw (2+) to[out=-30] (10+);
 \draw (2+) to[out=-40] (6+);
 \draw (4+) to[out=-40, in=180] (12+);
 \draw (4+) to[out=-50, in=180] (20+);
 \draw (4+) to[out=-60, in=180] (44+);
 \draw (6+) -- (3+) -- (6+3) -- (12+);
 \draw (10+) -- (5+) -- (10+5) -- (20+);
 \draw (22+) -- (11+) -- (22+11) -- (44+);
 \draw (12+3) to[bend right] (6+3);
 }
 \end{tikzpicture}
 \caption{Conjugacy classes with $2$-adically annihilated Hauptmoduln and their power maps.} \label{fig:2web}
\end{figure}
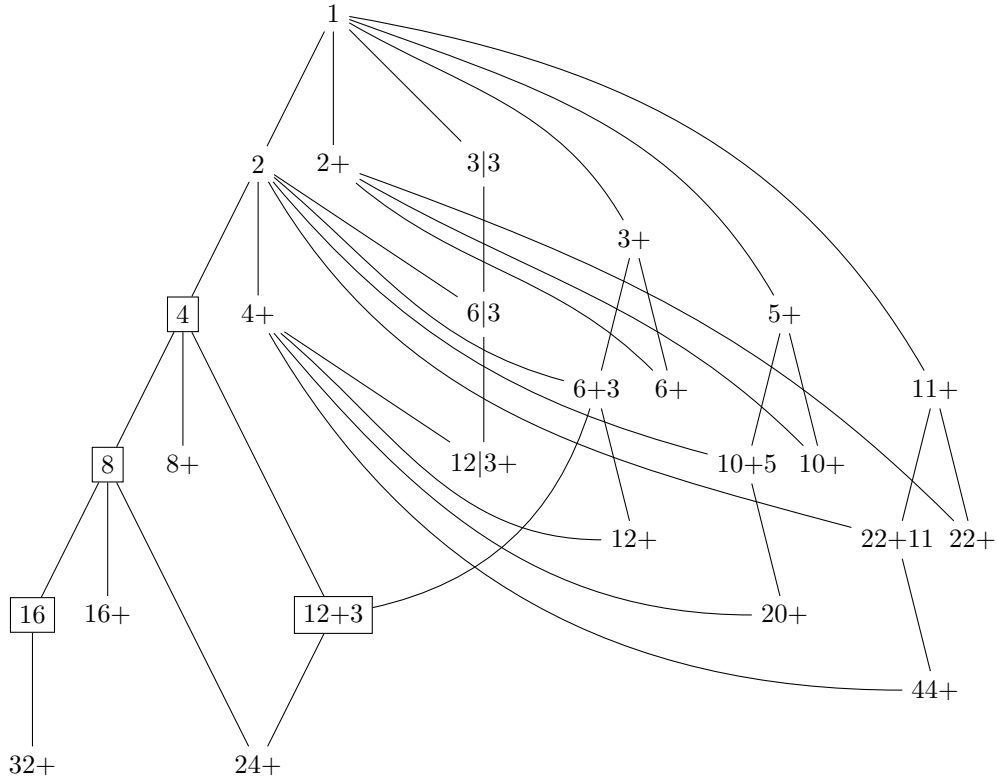

Finally, we consider which finite groups have infinite-dimensional representations with similar $p$-adic properties. We define a \textit{moonshine module} for a finite group $G$ to be a graded $G$-module $V = \bigoplus_{n=-1}^\infty V_n$ such that for each $g\in G$ the graded trace $\mathcal T_g = \sum \tr(g|V_n)q^n$ associated to the action of $g$ on $V$ is the Hauptmodul of an order $\ord(g)$ conjugacy class of the monster. We also require that the power maps of $G$ interact with the Hauptmoduln in a way that mimics what occurs in monstrous moonshine; see Section~\ref{sec:moonshine} for the precise condition. For an irreducible character $\chi$ of $G$, we write $m_\chi(n)$ for the multiplicity of $\chi$ appearing in the character of $G$ acting on $V_n$, and define the \textit{multiplicity generating function}
\begin{gather*}
\mathcal M_\chi(\tau) = \sum_{n=-1}^\infty m_\chi(n)q^n.
\end{gather*}
These series $\mathcal{M}_\chi$ were perhaps first studied in \cite{Harada96}.
We say that a moonshine module $V$ for $G$ is a \textit{$p$-adic moonshine module} if $\mathcal M_\chi$ is $p$-adically annihilated for each irreducible character $\chi$. We may then ask various questions about finite groups with $p$-adic moonshine modules, such as the number of such groups and which primes may divide their orders. In Section~\ref{sec:moonshine} we address these questions and give examples of groups with $p$-adic moonshine modules. In particular, we find that the groups in Table~\ref{tbl:intro_monster_subgroups} have $p$-adic moonshine for the listed $p$ in a slightly more general sense explained in Section~\ref{ssec:M-monster}. These groups arise as the centralizers of certain commuting pairs of elements of the monster in the conjugacy class $pA$. For other instances of moonshine modules for centralizers of elements of the class $pA$, see Ryba's modular moonshine conjectures \cite{Ryba96}, which were proved by Borcherds and Ryba \cite{Borcherds98, Borcherds96}.

\begin{table}[ht]
\centering
\begin{tabular}{ c|c|c|c|c|c }
 $p$ & $2$ & $3$ & $5$ & $7$ & $11$ \\ \hline
 $C\big(pA^2\big)$ & $2^2 \cdot {}^2\!E_6(2)$ & $3^2 \times O_8^+(3)$ & $5^2 \times U_3(5)$ & $7^2 \times L_2(7)$ & $11^2$ \tsep{3pt}\bsep{3pt}\\ \hline
 $\#C\big(pA^2\big)$ & $2^{38} {\cdot} 3^9 {\cdot} 5^2 {\cdot} 7^2 {\cdot} 11 {\cdot} 13 {\cdot} 17 {\cdot} 19$ & $2^{12} {\cdot} 3^{14} {\cdot} 5^2 {\cdot} 7 {\cdot} 13$ & $2^4 {\cdot} 3^2 {\cdot} 5^5 {\cdot} 7$ & $2^3 {\cdot} 3 {\cdot} 7^3$ & $11^2$\tsep{3pt}
\end{tabular}
\caption{Subgroups of the monster with weakly $p$-adic moonshine.}\label{tbl:intro_monster_subgroups}
\end{table}

Before proceeding, we outline the structure of this paper. We begin in Section~\ref{sec:technical} by proving technical lemmas that will be useful later in the paper. In Section~\ref{sec:serre}, we extend Serre's theory of $p$-adic modular forms such that it becomes applicable to the groups appearing in monstrous moonshine, and we begin to see $p$-adic properties of Hauptmoduln. In Section~\ref{sec:annihilation}, we prove Theorems~\ref{thm:annihilation} and~\ref{thm:power} using both the theory of Section~\ref{sec:serre} and the interplay between power maps and $p$-adic properties. We conclude in Section~\ref{sec:moonshine} by considering finite groups with $p$-adic moonshine modules, and showing that only finitely many such groups exist. We also discuss examples of groups with $p$-adic moonshine, including those in Table~\ref{tbl:intro_monster_subgroups}.

\section{Preliminaries}\label{sec:technical}
In this section, we collect technical details and definitions that will be used later. We first describe the types of groups $\G$ whose Hauptmoduln will be studied. We then discuss various properties of operators on spaces of modular forms, most importantly the $U_p$ operator and the Atkin--Lehner involutions. We also give descriptions of which cusps a Hauptmodul may have poles at once $U_p$ is applied to it, and we give a modular form $g$ with zeros at all such cusps. Finally, we discuss the trace of a modular form, which transforms modular forms on some $\G$ into modular forms on some $\G'\ge \G$. These facts will ultimately be used to interpret Hauptmoduln as $p$-adic modular forms in Section~\ref{sec:serre}.

\subsection[$n|h$-type groups]{$\boldsymbol{n|h}$-type groups}\label{ssec:T-nhtype}
Monstrous moonshine associates to each $g\in \M$ a Hauptmodul $\mt_g$ for some genus zero group $\G_g\le \GL_2(\R)^+$. This means that $\G_g\lmod \H^*$ is a genus zero curve and that $\mt_g$ is a generator for the function field such that $\mt_g$ is bounded away from the cusp $\infty$; moreover the $q$-expansion of $\mt_g$ at infinity begins $\mt_g = q^{-1} + O(q)$. Conway and Norton described the groups $\G_g$ in \cite{Conway79}, all of which take on a particular form which we reproduce here.

First we describe the normalizer of $\G_0(N)$ in $\PSL_2(\R)$. Let $h$ be the largest integer such that $h^2|N$ and $h|24$, and set $n = N/h$. The normalizer of $\G_0(N)$ is given by $\bigcup_{e\exact n/h} w_e$ where $w_e$ is the set of all matrices $A = \left(\begin{smallmatrix}ae & b/h \\ cn & de\end{smallmatrix}\right)$ such that $a,b,c,d\in \Z$ and $\det A = e$. Here the notation $x\exact y$ means that $x$ \textit{exactly divides}~$y$, i.e., that $x|y$ and $\gcd(x,y/x) = 1$. Given integers $e_1$, $e_2$ we set $e_1*e_2 = \frac{e_1e_2}{\gcd(e_1,e_2)^2}$, and under $*$ the set of exact divisors of any integer~$N$ forms the abelian group $(\Z/2\Z)^n$ where $n$ is the number of primes dividing $N$.

More generally, a class of subgroups called \textit{$n|h$-type groups} is defined as follows. Let $n$ be any positive integer and let $h|\gcd(n,24)$. Set $N=nh$ and $w_e$ as above, for $e\exact n/h$. We define the group $\G_0(n|h) = w_1$. We will often abuse notation and write $w_e$ for any element of $w_e$, and we see that $w_{e_1}w_{e_2} = w_{e_1*e_2}$. Since $h|24$ we have that $m^2\equiv 1\Mod{h}$ for all $m$ coprime to $h$. For a subgroup $\{1,e_1,\dots, e_n\}$ of the group of exact divisors of $n/h$, we then define
\begin{gather*}\Gamma_0(n|h){+}e_1,e_2,\dots,e_n = \langle \Gamma_0(n|h),w_{e_1},w_{e_2},\dots,w_{e_n}\rangle = w_1\cup w_{e_1}\cup w_{e_2}\cup\dots\cup w_{e_n}.\end{gather*}
A group of this form is called an $n|h$-type group.

Setting $N = nh$, the group $\G_0(n|h){+}e_1,\dots, e_n$ normalizes both $\G_0(n|h)$ and $\G_0(N)$, and the~$w_{e_i}$ are cosets of $\G_0(n|h)$. When $h = 1$ we have $\G_0(n|1) = \G_0(N)$ and we denote the matrix
\begin{gather*}w_e = \Matrix{ae & b \\ cN & de}\end{gather*}
by $W_e$. The matrices $W_e$ for $e\exact N$ are called \textit{Atkin--Lehner involutions}. Given an Atkin--Lehner involution $W_E$ on $\G_0(N)$, we can interpret this as an element of $\G_0(n|h)$ via
\begin{gather*}W_E = \Matrix{aE & b \\ cN & dE} = \Matrix{aE/h_E & b/h_E \\ cN/h_E & dE/h_E} = w_e,\end{gather*}
where $h_E$ is the largest divisor of $h$ with $h_E^2|E$ and $e = E/h_E^2$. In fact, setting
\begin{gather*}\AL(\G) = \{e\exact nh \colon W_e\in \G\text{ and every prime dividing }e\text{ also divides }n/h\}\end{gather*}
we have that this association gives a bijection
\begin{gather*}\AL(\G)\longleftrightarrow\{e\colon w_e\subset \G\}.\end{gather*}
For example, letting $\G = \G_0(8|2){+}4$ we have
\begin{gather*}\AL(\G) = \{1,16\}\longleftrightarrow \{1,4\} = \{e\colon w_e\subset \G\}.\end{gather*}

When dealing with $n|h$-type groups, it is standard to simplify notation in the following ways. When $h = 1$, we simply omit the $|h$, so that $\G_0(n|1) = \G_0(n)$, and when all $e\exact n/h$ are included in a group, we simply write $\G_0(n|h){+}$ so that $\G_0(8|2){+} = \G_0(8|2){+}4$. We will also sometimes use the symbol $n|h{+}e,f,\dots$ to represent the group $\G_0(n|h){+}e,f,\dots$ in order to save space, particularly in tables, so we might write $8|2{+}$ instead of $\G_0(8|2){+}$.

The $n|h$-type groups appear in monstrous moonshine as \textit{eigengroups} of the Hauptmoduln. That is, the Hauptmodul $\mt$ has an associated group $\G_0(n|h) {+} e,\dots$ such that for all $A$ in this group, $\mt(A\tau) = \mu \mt(\tau)$ for some $\nth{h}$ root of unity $\mu$. Conway--Norton \cite{Conway79} conjectured the following rule for computing the eigenvalue $\lambda$ corresponding to an element of $\G_0(n|h){+}e,\dots$:
\begin{enumerate}\itemsep=0pt
 \item[(i)] $\lambda = 1$ for any element of $\G_0(N)$,
 \item[(ii)] $\lambda = 1$ for all $W_e$ with $\al{e}{\G},$
 \item[(iii)] $\lambda = {\rm e}^{-2\pi {\rm i} /h}$ for the element $x = \left(\begin{smallmatrix}1 & 1/h \\ 0 & 1\end{smallmatrix}\right)$,
 \item[(iv)] $\lambda = {\rm e}^{\pm 2\pi {\rm i} /h}$ for the element $y = \left(\begin{smallmatrix} 1 & 0\\ n & 1\end{smallmatrix}\right)$ where the sign is $+$ if and only if $\tau\mapsto \frac{-1}{N\tau}$ is in $\G_0(n|h){+}e,\dots$.
\end{enumerate}
This rule's well-definedness and correctness follow from~\cite{Ferenbaugh93} and the correctness of the monstrous moonshine conjectures, respectively.

Since the cosets $x$ and $y$ generate $\G_0(n|h)$, we have $\G = \langle x,y,W_e \colon \al{e}{\G}\rangle$ for any $n|h$-type group $\G$. Hence Conway--Norton's rule uniquely determines a map $\lambda\colon \G\rightarrow \bm\mu_{h}$ where $\bm\mu_{h}$ denotes the group of~$\nth{h}$ roots of unity. We always use $\lambda$ to denote this map.

More generally, let $\G$ be an $n|h$-type group. An \textit{eigenvalue map} is a homomorphism \mbox{$\eta\colon \!\G\!\rightarrow\! \bm\mu_{2h}$} such that $\G_0(nh)\subset \ker \eta$. Then we define
\begin{gather*}\G_\eta = \ker \eta.\end{gather*}
When $\lambda$ is the map given by Conway--Norton's rule, we have that $\G_\lambda$ is an index $h$ subgroup of $\G$ called the \textit{fixing group} of $\G$. However, Conway--Norton's rule does not always give a well-defined map, so $\G_\lambda$ does not exist for every $n|h$-type group~$\G$; for example $\G_{\lambda}$ doesn't exist when $\G=\Gamma_0(21|3)$. Ferenbaugh \cite[Theorem~2.8]{Ferenbaugh93} classified the groups $\G$ for which Conway--Norton's rule is consistent, and we call such $\G$ \textit{admissible}. There are $213$ admissible $n|h$-type groups giving genus zero groups, including all $171$ groups appearing in monstrous moonshine. Ferenbaugh also determined the structure of the quotient~$\G/\G_0(nh)$, and therefore also the structure of $\G_\lambda/\G_0(nh)$. In $174$ cases, including the groups of monstrous moonshine, the latter quotient group has exponent $2$; the remaining $3$ groups are known as the ``ghosts''. For further discussion of which $n|h$-type groups appear in monstrous moonshine, see~\cite{Conway04}.

In the next section we study modular forms on $n|h$-type groups with given eigenvalue maps, and the action of the $U_p$ operator on such spaces of modular forms.

\subsection[Action of $U_p$ on Hauptmoduln]{Action of $\boldsymbol{U_p}$ on Hauptmoduln}\label{ssec:T-operators}
Given an $n|h$-type group $\G$ with eigenvalue map $\eta\colon \G\rightarrow\bm\mu_{2h}$, we say a weight $k$ weakly holomorphic modular form on $\G_0(nh)$ is \textit{on $\G$ with eigenvalue map $\eta$} or \textit{on $(\G,\eta)$} if
\begin{gather*}f|_k\gamma = \eta(\gamma)f\qquad\text{for all} \quad \gamma\in \G,\end{gather*}
where the weight $k$ slash operator is defined by (\ref{eqn:slash_def}) below. By a \textit{weakly holomorphic modular form} we mean a meromorphic modular form whose poles are supported on the cusps; on the other hand a \textit{modular form} is assumed to be holomorphic everywhere. We write $M_k(\G,\eta)$ for the space of weight $k$ modular forms on $\G$ with eigenvalue map $\eta$. Similarly, we denote the space of weight $k$ modular forms on $\G_0(nh)$ invariant under all $\g\in \G_\eta$ by $M_k(\G_\eta)$. Throughout, all our weights will be integers.

Fix a prime $p$. In this section, we study $U_p$ applied to Hautpmoduln $\mt$, and more generally weakly holomorphic modular forms on $\G_\eta$ or on $\G$ with eigenvalue map $\eta$. For our results to extend to $n|h$-type groups, the results of this section will be stated in the necessary general language. However the reader looking to use these results for modular forms on $\G_0(N){+}e,\dots$ should remember that this corresponds to taking $h = 1$ and ignoring eigenvalue maps in the following results.

Recall that the weight $k$ slash operator $|_k\gamma$ for $\gamma\in \GL_2^+(\R)$ is defined by
\begin{gather}\label{eqn:slash_def}
(f|_k\gamma)(\t) = (\det \g)^{k/2}(c\t + d)^{-k}f(\g\t).
\end{gather}
 If $f$ is on $\SL_2(\Z)$ then for $N\in \N$, we have $f(N\tau)$ on $\G_0(N)$, and for $d|N$ and $e\exact N$,
\begin{gather}\label{eqn:exact_divisor_action}
f(d\t)|_kW_e = \left(\frac{d*e}{d}\right)^{k/2}f((d*e)\t).
\end{gather}
In terms of the slash operator, $U_p$ is defined on weight $k$ modular forms by
\begin{gather}\label{eqn:up_defn}
f|U_p = p^{k/2 - 1}\sum_{\mu = 0}^{p-1}f|_kS_\mu = \frac{1}{p}\sum_{\mu = 0}^{p-1}f(S_\mu\tau),
\end{gather}
where $S_\mu=\left(\begin{smallmatrix} 1 & \mu \\ 0 & p\end{smallmatrix}\right)$. This operator is independent of $k$ and acts on Fourier expansions by
\begin{gather*}\left(\sum a(n)q^n\right)|U_p = \sum a(pn)q^n.\end{gather*}

We first recall some basic facts about the $U_p$ operator (see \cite[Section 2]{Lehner70}).

\begin{Lemma}\label{lemma:up_basic_facts} Let $f$ be a weight $k$ meromorphic modular form on $\G_0(p^\alpha N)$ where $p\nmid N$.
 \begin{enumerate}\itemsep=0pt
 \item[$(a)$] If $e\exact N$ then $f|U_p|_kW_e = f|_kW_e|U_p$.
 \item[$(b)$] $f|U_p$ is modular on $\G_0(p^\beta N)$ where $\beta = \max\{1,\alpha - 1\}$.
 \end{enumerate}
\end{Lemma}

 The following lemma extends these facts from $\G_0(N)$ to $n|h$-type groups.
\begin{Lemma}\label{lemma:up_nh_facts} Let $p\nmid nh$ be prime, and let $f$ be a weight $k$ meromorphic modular form on a~$p^\alpha n|h$-type group $\G$ with eigenvalue map $\eta$.
 \begin{enumerate}\itemsep=0pt
 \item[$(a)$] Let
 \begin{gather*}x = \Matrix{1 & 1/h \\ 0 & 1}\qquad\text{and}\qquad y = \Matrix{1 & 0 \\ p^\alpha n & 1}\end{gather*}
 be the matrices of Section~{\rm \ref{ssec:T-nhtype}}. Then
 \begin{align*}
 f|U_p|_kx^p &= f|_kx|U_p \qquad \text{and} \qquad f|U_p|_ky^p = f|_ky|U_p.
 \end{align*}
 \item[$(b)$] Let $\G'$ be the $p^\beta n|h$-type group such that $\beta = \max\{1,\alpha - 1\}$ and $\al{e}{\G'}$ if and only if $\al{e}{\G}$ and $p\nmid e$. Let $\eta'$ be an be an eigenvalue map on $\G'$ such that
 \begin{gather*}\eta'(x^p) = \eta(x),\qquad \eta'(y^p) = \eta(y),q\quad \eta'(W_e) = \eta(W_e)\qquad \text{for} \quad \al{e}{\G'}.\end{gather*}
 Then $f|U_p$ is on $\G'$ with eigenvalue map $\eta'$.
 \end{enumerate}
\end{Lemma}
\begin{proof}For $0\le \mu\le p-1$, let $S_\mu = \left(\begin{smallmatrix}1 & \mu \\ 0 &p\end{smallmatrix}\right)$ denote the matrix appearing in the definition~(\ref{eqn:up_defn}) of $U_p$. The first identity of part (a) follows from the equation $S_\mu x^p = xS_\mu$. The second identity follows from the equation
 \begin{gather}\label{eqn:y_commute}S_\mu y^p = \begin{pmatrix}* & * \\ p^\alpha n(-1 + p^2 + p^{\alpha+1} n\mu) & * \end{pmatrix}y S_\mu,\end{gather}
 where each $*$ is an integer such that the matrix has determinant $1$. Since $p\nmid h$, we have that $p^2\equiv 1\pmod{h}$ (since $h$ is a divisor of $24$), so the matrix appearing in (\ref{eqn:y_commute}) is in $\G_0(p^\alpha nh)$ and therefore fixes $f$.

 For part (b), note that since $p\nmid h$ and $\beta\le \alpha + 1$, the matrices $x^p$ and $y^p$ generate $\G_0\big(p^\beta n|h\big)$. Thus part (b) follows from part (a) and Lemma~\ref{lemma:up_basic_facts}.
\end{proof}

Let $\sigma_a\colon \mu\mapsto\mu^a$ be an endomorphism of $\bm\mu_{2h}$. We set $\eta^{\sigma_a} = \sigma_a \circ \eta$. The preceding lemma says that if $\G$ is an $pn|h$-type group with $p\nmid nh$ then $U_p$ is a map $M_k(\G,\eta)\rightarrow M_k(\G,\eta^{\sigma_p})$.

Analogous to the decomposition $M_k(\Gamma_1(N)) = \bigoplus_\chi M_k(\Gamma_0(N),\chi)$ over mod $N$ Dirichlet characters $\chi$, we have the following decomposition of $M_k(\G_\eta)$.
\begin{Lemma}\label{lemma:eigenmap_decomp} Let $\G$ be an $n|h$-type group with eigenvalue map $\eta$, such that $\im \eta = \bm\mu_{h'}\le \bm\mu_{2h}$. There is a decomposition
 \begin{gather*}M_k(\G_\eta) = \bigoplus_{\substack{\eta'\colon \G\rightarrow \bm\mu_{2h}\\ \ker\eta'\supseteq\G_\eta}} M_k(\G,\eta') = \bigoplus_{a = 0}^{h' - 1} M_k\big(\G,\eta^{\sigma_{a}}\big).\end{gather*}
\end{Lemma}
\begin{proof} Since $\G/\G_\eta \cong \im\eta$ is finite and abelian, the action of $\G$ on $M_k(\G_\eta)$ can be simultaneously diagonalized.
\end{proof}

Let $\G$ be a $p^\alpha n|h$-type group with eigenvalue map $\eta$ such that $\G_\eta$ is genus zero and \mbox{$p\nmid nh$}. Lemma~\ref{lemma:up_nh_facts} tells us under which group $\mt|U_p$ is modular. We know that $\mt|U_p$ is weakly holomorphic, and next characterize the cusps at which it may be unbounded. We first recall that a set of representatives for the cusps of $\G_0(N)$ is given by
\begin{gather*}\left\{\frac{a}{b} \colon b|N,a\in\Z\right\}/{\sim}, \qquad\text{where}\quad \frac{a}{b}\sim\frac{c}{d}\iff b = d \quad \text{and} \quad a\equiv c\Mod{\gcd(b,N/b)}.\end{gather*}
Moreover, if two cusps $\frac{a}{b}$ and $\frac{c}{d}$ (not necessarily representatives of the form above) are equivalent under $\G_0(N)$ and $\gcd(a,b) = \gcd(c,d) = 1$, then $\gcd(b,N) = \gcd(d,N)$. Both of these facts follow from \cite[Proposition~3.8.3]{Diamond_Shurman}.

In what follows, if $\G\le \GL_2^+(\R)$ is commensurable with $\SL_2(\Z)$ and $s,s'\in\mathbb{P}^1(\Q)$ are cusps, then we write
\begin{gather*}s\stackrel{\G}{\sim}s'\end{gather*}
to mean that $s$ and $s'$ are equivalent under $\G$. If $\G = \G_0(N)$ then we simply write $s\stackrel{N}{\sim}s'$.

\begin{Lemma}\label{lemma:hauptmodul_poles} Let $p\nmid nh$ be prime, and let $\G$ be a $p^\alpha n|h$-type group with eigenvalue map~$\eta$. Let~$\G'$ and~$\eta'$ be the $p^\beta n|h$-type group and eigenvalue map defined in Lemma~{\rm \ref{lemma:up_nh_facts}(b)}. Suppose $f$ is a weakly holomorphic modular form on $\G$ with eigenvalue map $\eta$, so that $f|U_p$ is on $\G'$ with eigenvalue map $\eta'$ by Lemma~{\rm \ref{lemma:up_nh_facts}}.
 \begin{enumerate}\itemsep=0pt
 \item[$(a)$] If $f$ is bounded away from $\infty$ then the poles of $f|U_p$ are supported on the cusps
 \begin{gather*}\{\infty\}\cup \big\{s \colon s\stackrel{p}{\sim}0\big\}\quad \text{of} \quad \G'_{\eta'}.\end{gather*}
 \item[$(b)$] Suppose that $p\nmid e$ for all $\al{e}{\G}$. If the poles of $f$ are supported on $\{\infty\}\cup\{s \colon s\stackrel{p}{\sim}0\}$ then the same is true of~$f|U_p$.
 \item[$(c)$] Suppose that $\alpha\ge 2$ and $p\nmid e$ for all $\al{e}{\G}$. If $f$ is bounded away from $\infty$ and $f|U_p$ is bounded at $\infty$, then $f|U_p$ is constant.
 \end{enumerate}
\end{Lemma}
\begin{proof} We begin with part (a). Suppose that $f$ is bounded away from $\infty$, and suppose that~$f|U_p$ is unbounded at the cusp $s$. If $s\stackrel{p}{\sim}0$ then we are done. Thus we may assume that $s\stackrel{p}{\sim}\infty$. Also, by (\ref{eqn:up_defn}), we must have $S_\mu\cdot s\stackrel{\G_\eta}{\sim}\infty$ for some $\mu$. Equivalently, $s\stackrel{p^\alpha nh}{\sim}S_\mu^{-1}A\cdot \infty$ for some $A$ that can be expressed as a word in the matrices
 \begin{gather*}x = \Matrix{1 & 1/h \\ 0 & 1},\qquad y = \Matrix{1 & 0 \\ p^\alpha n & 1}, \qquad \text{and}\qquad W_e\in \G_\eta\end{gather*}
 of Section~\ref{ssec:T-nhtype}. Since $W_e$ normalizes $\Gamma_0(p^\alpha n|h)$ (see \cite[Theorem~2.7]{Ferenbaugh93}), we may write all of the~$W_e$'s on the right of~$A$. By cancelling $W_{p^\alpha}$'s, we can moreover demand that $p\nmid e$ for all $W_e$ appearing in $A$, for otherwise we would have $A\cdot \infty \stackrel{p}\sim 0$ while $S_\mu\cdot s\stackrel{p}{\sim}\infty$. As in the proof of Lemma~\ref{lemma:up_nh_facts}, we can write $S_\mu^{-1}A = A'S_\mu^{-1} V$ for some $V\in \G_0(p^\alpha nh)$, where $A'$ is obtained from~$A$ by replacing each $x$ with $x^p$ and each $y$ with $y^p$. Hence $s\stackrel{p^\alpha nh}{\sim}A'S_\nu^{-1}\cdot \infty$.

If $\alpha\le 1$ and $\beta = 1$ then either $S_\nu^{-1}\cdot \infty\stackrel{p^\beta nh}{\sim}\infty$ or $\frac{1}{nh}$. But since $A'\cdot \frac{1}{nh}\stackrel{p}{\sim}0$, we must have $S_\nu^{-1}\cdot\infty\stackrel{p^\beta nh}{\sim}\infty$. Hence $s\stackrel{p^\beta nh}{\sim}A'\cdot\infty \stackrel{\G'_{\eta'}}{\sim}\infty$, as desired. The case $\alpha\ge 2$ is dealt with similarly, completing the proof of~(a).

 A similar argument gives part (b), where we must now use the fact that $p\nmid e$ for $\al{e}{\G}$ to show that $p\nmid e$ for all $W_e$ appearing in $A$. For part (c), one finds that $f|U_p$ may only have a~pole at the cusp $\infty$; however, since it does not have a pole here by hypothesis, $f|U_p$ is bounded everywhere and hence constant.
\end{proof}

\begin{Remark}\label{rmk:hauptmodul_killed} Lemma~\ref{lemma:hauptmodul_poles}(c) delivers a class of Hauptmoduln $\mathcal{T}$ for which $\mathcal{T}|U_p=0$. For the Hauptmoduln from monstrous moonshine to which this applies, this property also follows from \cite[Lemma~3.2]{Larson16}. Furthermore, if $\G$ is a $n|h$-type group and $\eta$ is an eigenvalue map with $\eta(x) = {\rm e}^{2\pi {\rm i} m/h}$, then any meromorphic modular form $f=\sum\limits_{n=m}^\infty a(n) q^n$ on $\G$ with eigenvalue map~$\eta$ has $a(n)=0$ if $n\not\equiv m\Mod{h}$, since $x$ sends $q\mapsto {\rm e}^{2\pi {\rm i}m/h} q$.
 In particular, if $h\equiv0\Mod{p}$ and $\mt$ is the Hautpmodul on $\G_\lambda$, then $f|U_p = 0$.

Inspection of Table~\ref{tbl:full_annihilation} shows that the only monster Hauptmoduln with $\mt|U_p = 0$ are those with $h\equiv0\Mod{p}$ and those coming from Lemma~\ref{lemma:hauptmodul_poles}(c).
\end{Remark}

In Section~\ref{sec:serre} we will need a modular form $g$ on an $n|h$-type group $\G$ such that the zeros of $g$ can cancel the poles of $\mt|U_p$, whose locations were just determined. We will also need $g$ to have certain properties modulo $p$.

To construct $g$, we will use the \textit{modular discriminant} $\Delta(\t) = q\prod\limits_{n\ge 1}(1 - q^n)^{24}$. If a modular form $f=\sum\limits_{n=m}^\infty a(n) q^n$ has rational coefficients, we set $v_p(f)=\inf\limits_{n} v_p(a_n)$ where $v_p(a_n)=\sup\{r\in\Z\colon p^r\divides a_n \}$.
\begin{Lemma} \label{lemma:delta_quotient} Let $\G$ be a $pn|h$-type group where $p\nmid nh$ is prime and $p\nmid e$ for all $\al{e}{\G}$. Let $m = \#\AL(\G)$ and set $a = 12m(p-1)$. Then there is a modular form $g$ on $\G$ of weight $a$ with rational coefficients such that
 \begin{enumerate}\itemsep=0pt
 \item[$(a)$] $g \equiv 1 \Mod{p}$,
 \item[$(b)$] $v_p(g |_a W_p) \ge 6m(p + 1)$,
 \item[$(c)$] As a function on $\G_0(pnh)$, $g$ vanishes to order $\ge m\big(p^2 - 1\big)$ at every cusp equivalent to $0$ under $\G_0(p)$.
 \end{enumerate}
 Specifically, $g$ can be chosen to be
 \begin{gather*}\prod_{\al{e}{\G}}\left.\frac{\Delta(h\t)^p}{\Delta(ph\t)}\right|_{12(p-1)}W_e.\end{gather*}
\end{Lemma}
\begin{proof}First, let
 \begin{gather*}g_p(\t) = \frac{\Delta(h \t)^p}{\Delta(p h \t)}\end{gather*}
 and $a_p = 12(p-1)$. Note that $g_p(\tau)$ is invariant under $\G_0(pn|h)$. For any $\al{e}{\G}$
 \begin{gather}\label{eqn:g_hit_with_we}
 (g_p |_{a_p} W_e) (\t)
 = \frac{(\Delta(h \t) |_{12} W_e)^p}{\Delta(p h \t) |_{12} W_e}
 = \left(\frac{h*e}{h}\right)^{6(p-1)}\frac{\Delta((h*e)\t)^p}{\Delta(p(h*e)\t)}
 \end{gather}
 by (\ref{eqn:exact_divisor_action}). On the other hand, we see that
 \begin{gather}\label{eqn:g_hit_with_wp}
 (g_p|_{a_p}W_{pe})(\t)= p^{6(p + 1)}\left(\frac{h*e}{h}\right)^{6(p-1)}\frac{\Delta(p(h*e)\t)^p}{\Delta((h*e)\t)}.
 \end{gather}
 Since $\Delta$ is nonzero on $\H$, (\ref{eqn:g_hit_with_we}) shows that $g_p|_{a_p}W_e$ is a modular form on $\G_0(pnh)$ with rational coefficients. Moreover, (\ref{eqn:g_hit_with_wp}) shows that each $g_p|_{a_p}W_{pe}$ for $\al{e}{\G}$ vanishes to order \mbox{$(h*e)\big(p^2 - 1\big)$} at~$\infty$, so that each $g_p|_{a_p}W_e$ vanishes to order $\ge p^2 - 1$ at the cusps $s\stackrel{p}{\sim}0$.

 Since $(1 - q^n)^p\equiv 1 - q^{np}\Mod{p}$, we see that $g_p|_{a_p}W_e\equiv \big(\frac{h*e}{h}\big)^{6(p-1)}\equiv 1\Mod{p}$. Moreover $v_p(g_p |_{a_p} W_p) = 6(p+1)$. Thus, we may set
 \begin{gather*}g = \prod_{\al{e}{\G}} g_p |_{a_p} W_e,\end{gather*}
 which clearly satisfies the conditions given.
\end{proof}

\begin{Remark}The function in Lemma \ref{lemma:delta_quotient} is chosen for its large order zeroes, in contrast with symmetrizations of the function $g = E_a - p^{a/2}E_a|_aW_p$ of \cite[Lemma~8]{Serre72}. This will be computationally useful in Section~\ref{sec:annihilation}.
\end{Remark}

We conclude with a few tools for working with $q$-expansions of mod $p$ modular forms.

\begin{Lemma}[Sturm's bound \cite{Sturm87}]\label{lemma:sturm}
Let $f\in M_k(\Gamma_0(N))$ with integer coefficients $a_n$. If $p\divides a_n$ for $p\leq (k/12)[\SL_2(\Z) : \Gamma_0(N)]$, then $p\divides a_n$ for all~$n$ .
\end{Lemma}

We will apply Sturm's bound to weakly holomorphic modular forms after multiplying by a~power of the function from Lemma~\ref{lemma:delta_quotient}. We thus bound the pole orders of $\mt|U_p$.

\begin{Lemma}\label{lemma:pole_bound}Let $f$ be a weakly holomorphic function on $X_0(p^\alpha N)$, where $p\nmid N$, and let $\beta=\max\{\alpha,1\}$. If $r$ is the maximum order of a pole of $f$ on $X_0(p^\alpha N)$, then the poles of $f|U_p$ as a~function on $X_0\big(p^{\beta} nh\big)$ have order at most $rp^2$ when $\alpha=0$, and order at most $rp$ otherwise.
\end{Lemma}

\begin{proof}The ramification index of each cusp of $X_0(pnh)$ over $X_0(nh)$ is a divisor of $p$. Thus, for the case $\alpha=0$, the maximum order of a pole of $f$ pulled back to $X_0(pnh)$ is $rp$. The $U_p$ operator may be defined via the correspondence
\begin{center}
 \begin{tikzcd}
 \left(\G_0(p^{\beta}nh)\cap\gamma^{-1}\G_0(p^{\beta}nh)\gamma\right)\lmod\H^*\ar{d}\ar{r}{\sim} & \left(\gamma\G_0(p^{\beta}nh)\gamma^{-1}\cap\G_0(p^{\beta}nh)\right)\lmod\H^*\ar{d}\\
 X_0(p^{\beta}nh) &X_0(p^{\beta}nh)
 \end{tikzcd}
\end{center}
where $\gamma = \left(\begin{smallmatrix}1 & 0 \\ 0 &p\end{smallmatrix}\right)$. The projections have degree $p$, so $\mt$ pulls back to a function on $\big(\gamma\G_0\big(p^{\beta}nh\big)\gamma^{-1}$ $\cap\G_0(p^{\beta}nh)\big)\lmod\H^*$ with poles at most degree $rp^2$ when $\alpha=0$ and $rp$ else. The other maps of the correspondence, which are pullback by the isomorphism and trace down to $X_0(p^{\beta}nh)$, do not increase the maximum pole order, so the claim follows.
\end{proof}

\subsection{Trace formulas}\label{ssec:T-traces}
 Following Serre's idea \cite{Serre72}, we will apply the trace to view classical modular forms as $p$-adic modular forms of lower level. In this section we discuss a few properties of trace maps.

Suppose $\G$ and $\G'$ are subgroups of $\GL_2^+(\R)$ with $\G$ a finite index subgroup of $\G'$. We define the \textit{trace} $\tr_{\G\lmod \G'}$ from $\G$ to $\G'$ to be the operation
\begin{gather}\label{eqn:trace_def}
 \tr_{\G\lmod \G'} f = \sum_{i=1}^m f |_k \g_i,
\end{gather}
where $\{\g_1, \dots, \g_m\}$ is a system of right coset representatives for $\G \lmod \G'$. If $f$ is modular on $\G$ then $\tr_{\G\lmod \G'}$ is modular on the larger group $\G'$. When $\G = \G_0(N)$ and $\G' = \G_0(N')$, we simply write $\tr_{N\lmod N'}$ for $\tr_{\G\lmod \G'}$.

First consider $\G = \G_0(pN)$ and $\G' = \G_0(N)$. The following generalizes \cite[Lemma~7]{Serre72}.
\begin{Lemma}\label{lemma:trace_basic}Let $p\nmid N$ be prime.
 \begin{enumerate}\itemsep=0pt
 \item[$(a)$] A set of representatives for $\G_0(pN)\lmod \G_0(N)$ is given by
 \begin{gather*}\left\{\Matrix{1 & 0 \\ 0 & 1}\right\}\cup\left\{\Matrix{1 & \lambda \\ N & N\lambda + 1} \colon 1\le \lambda\le p\right\}.\end{gather*}
 \item[$(b)$] If $f$ is a weight $k$ modular form on $\G_0(pN)$ then
 \begin{gather*}
 \tr_{pN\lmod N}f = f + p^{1 - k/2}(f|_kW_p)|U_p
 \end{gather*}
 where $W_p$ is the corresponding Atkin--Lehner involution on $\G_0(pN)$.
 \end{enumerate}
\end{Lemma}
\begin{proof} Since $[\G_0(N) : \G_0(pN)] = p+1$, part (a) follows upon checking that the representatives are inequivalent modulo $\G_0(pN)$. One can also check that for any $1\le \lambda\le p$, if $\mu \equiv N^{-1}\Mod{p}$, then
 \begin{gather*}\Matrix{1 & \lambda \\ N & N\lambda + 1} = VW_p\Matrix{1/p & (\lambda + \mu)/p \\ 0 & 1}\end{gather*}
 for some $V\in \G_0(pN)$. Part (b) follows from this.
\end{proof}

The remainder of this section will extend Lemma~\ref{lemma:trace_basic} to the more general context we need, for example tracing from $\G_0(pN){+}e,\dots$ to $\G_0(N){+}e,\dots$ for $p\nmid N$. More precisely, for a prime $p\nmid nh$, suppose that $\G$ is a $pn|h$-type group with eigenvalue map $\eta$ with $p\nmid e$ for all $\al{e}{\G}$. Let $\G'$ be the $n|h$-type group such that $\AL(\G') = \AL(\G)$. We have $\G\subset\G'$, and can take $\eta'$ to be the eigenvalue map on $\G'$ with $\eta'|\G = \eta$. Since $\G'$ is generated by
\begin{gather*}x = \Matrix{1 & 1/h \\ 0 & 1},\qquad y = \Matrix{1 & 0 \\ n & 1},\qquad W_e \ \text{such that} \quad e\in \AL(\G') = \AL(\G)\end{gather*}
as in Section~\ref{ssec:T-nhtype}, this uniquely determines $\eta'$. Then \cite[Lemma 2.3]{Ferenbaugh93} and \cite[Theorem 2.7]{Ferenbaugh93} together imply that the inclusion of representatives
\begin{gather*}\iota\colon \ \G/\G_0(pnh)\hookrightarrow \G'/\G_0(nh)\end{gather*}
is an isomorphism. We set $H = \im(\iota|\G_\eta)\le \G'$, and consider the restricted isomorphism
\begin{gather}\label{eqn:restricted_iso}
\iota|\G_\eta\colon \ \G_\eta/\G_0(pnh)\stackrel{\sim}{\hookrightarrow} H/\G_0(nh).
\end{gather}
We have that $H\le \G'_{\eta'}$. Moreover, $\im\eta' = \im\eta$, so $[\G':\G'_{\eta'}] = [\G:\G_\eta] = [\G':H]$ and thus $\G'_{\eta'} = H$. We have nearly proved the following lemma.
\begin{Lemma}\label{lemma:trace_facts} Let $\G$ be a $pn|h$-type group with eigenvalue map $\eta$ such that $p\nmid e$ for all $\al{e}{\G}$. Let $\G'$ be an $n|h$-type group with $\AL(\G') = \AL(\G)$ with eigenvalue map $\eta'$ such that $\eta'|\G = \eta$. Then for any weight $k$ modular form on $\G_\eta$ we have
 \begin{gather*}\tr_{\G_\eta\lmod \G'_{\eta'}}f = \tr_{pnh\lmod nh}f = f + p^{1 - k/2}(f|_kW_p)|U_p.\end{gather*}
\end{Lemma}
\begin{proof} We need to show $\tr_{\G_\eta\lmod \G_{\eta'}}f = \tr_{pnh\lmod nh}f$. Let $\{\gamma_i\}$ be any set of representatives for $\G_0(pnh)\lmod \G_0(nh)$. Then $\{\g_i\}$ is also a set of representatives for $\G_\eta\lmod \G'_{\eta'}$. Indeed, by the isomorphism~(\ref{eqn:restricted_iso}) we have $[\G'_{\eta'}:\G_\eta] = [\G_0(nh):\G_0(pnh)]$ so it suffices to check that no two $\g_i$ are equivalent. Suppose $\g_i\g_j^{-1}\in \G_\eta$. Then $\g_i\g_j^{-1}\in \G\cap \G_0(nh) = \G_0(pnh)$ so that $\g_i = \g_j$ as desired. The formula then follows from Lemma~\ref{lemma:trace_basic}.
\end{proof}

\begin{Remark}Lemma~\ref{lemma:trace_facts} only assumes that $f$ is on $\G_\eta$. Under the stronger assumption that~$f$ is on $(\G,\eta)$, we obtain the finer result that $\tr_{pnh\lmod nh}f$ is on $(\G',\eta')$. To see this, let $x$ and $Y = y^p$ be the generators of $\G_0(pn|h)/ \G_0(pnh)$, and choose appropriate representatives of~$W_e$ which normalize $\G_0(nh)$. Then, apply $x$, $Y$, and $W_e$ to both sides of~\eqref{eqn:trace_def}.
\end{Remark}

\section[$p$-adic modular forms]{$\boldsymbol{p}$-adic modular forms}\label{sec:serre}
In this section, we extend of Serre's theory of $p$-adic modular forms from \cite{Serre72} to Hauptmoduln and $n|h$-type groups. In particular, we study the interaction between eigenvalue maps and the mod $p$ weight filtration. These $p$-adic properties could be studied by extending the theory of Katz and others, but we choose to generalize Serre's original treatment in order to perform explicit calculations for our applications. Take $p\geq 5$, and let $\G$ be an $n|h$-type group with eigenvalue map $\eta$. We first study $M_k(\G,\eta)$ and its $p$-adic completion. In Section~\ref{ssec:S-hauptmoduln} we prove that Hauptmoduln $\mt$ become $p$-adic modular forms on some $(\G',\eta')$ under applications of $U_p$. In Section~\ref{ssec:S-hecke} we extract structural results concerning ordinary spaces and the action of $U_p$ on these $p$-adic modular forms. Again, readers interested only in modular groups of the form $\G_0(N){+}e,\dots$ can take $h = 1$ and let eigenvalue maps be identically $1$.

For an $n|h$-type group $\G$ with eigenvalue map $\eta\colon \G\rightarrow\bm\mu_{2h}$, we first define the spaces:
\begin{enumerate}\itemsep=0pt
 \item[(1)] $M_k^\Q(\G,\eta) = M_k(\G,\eta)\cap \psring{\Q}{q}$, the $\Q$-vector space of modular forms with rational $q$-expansion;
 \item[(2)] $M_k^{(p)}(\G,\eta) = M_k(\G,\eta)\cap\psring{\Z_{(p)}}{q}$, the $\Z_{(p)}$-module of modular forms with $p$-integral $q$-expansion; and
 \item[(3)] $\tilde M_k(\G,\eta) = M_k^{(p)}(\G,\eta) \otimes \F_p$, the $\F_p$-vector space obtained by reducing $M_k^{(p)}(\G,\eta)$ mod $p$.
\end{enumerate}
Similarly define $M_k^\Q(\G_\eta)$, $M_k^{(p)}(\G_\eta)$, and $\tilde M_k(\G_\eta)$. If $f$ reduces mod $p$ to a form in $\tilde{M}_k(\G,\eta)$, we abuse notation and write $f\in \tilde{M}_k(\G,\eta)$, and similarly for $\Gamma_\eta$. We focus on $M_k(\G,\eta)$, and the corresponding results for $M_k(\G_\eta)$ follow from Lemma~\ref{lemma:eigenmap_decomp}.

Following \cite{Serre72}, we define a \textit{$p$-adic modular form} on $(\G,\eta)$ to be a $q$-expansion $f\in\psring{\Q_p}{q}$ admitting a sequence $f_m\in M_{k_m}^\Q(\G,\eta)$ that converges $p$-adically with $v_p(f_m - f)\rightarrow\infty$, i.e.,
\begin{gather*}\lim_{m\rightarrow\infty} f_m = f.\end{gather*}
Similarly, a $p$-adic modular form on $\G$ is a $p$-adic modular form on $(\G,\bm{1})$.

For any $N$, if $f$, $f'$ are modular forms on $\G_0(N)$ of weight $k$, $k'$ and $f\equiv f'\Mod{p^n}$ then $k\equiv k'\Mod{p^{n-1}(p-1)}$ (see \cite[Corollary~4.4.2]{Katz73}). It follows that the weight of a $p$-adic modular form on $\G_0(N)$, defined as the limit of the $k_m$ in the space
\begin{gather*}\mf X = \varprojlim \Z/p^n\Z\times\Z/(p-1)\Z = \Z_p\times\Z/(p-1)\Z,\end{gather*}
exists and does not depend on the choice of sequence $(f_m)$. The same is true of $p$-adic modular forms on $(\G,\eta)$, since $M_k(\G,\eta)\subseteq M_k(\G_0(nh))$. In particular, forms in $M_k(\G,\eta)$ have trivial nebentypus, as required in~\cite{Katz73}.

The correct extension of the mod $p$ weight filtration to modular forms for $n|h$-type groups will feature a quadratic twist. To this end, we first twist our eigenvalue maps.

\begin{Definition}\label{def:twisted_eigmap}Let $\G$ be an $n|h$ type group and $p\nmid nh$ a prime. If $\eta$ is an eigenvalue map on~$\G$, then the \textit{twist} of $\eta$ is the map $\eta_t\colon \G\rightarrow \bm\mu_{2h}$ defined by
 \begin{gather*}
 \eta_t(\g) = \eta(\g)\qquad\text{for} \quad \g\in \G_0(n|h) \qquad\text{and}\qquad \eta_t(W_e) = \qr{e}{p}\eta(W_e)\qquad\text{for} \quad \al{e}{\G},
 \end{gather*}
 where $\qr{\cdot}{p}$ denotes the Legendre symbol.
\end{Definition}

\begin{Remark}If $\bm{1}$ denotes the trivial eigenvalue map $\bm{1}(\g) = 1$ for all $\g$, then $\eta_t = \eta\bm{1}_t$ for all $\eta$. Also, if $\qr{e}{p}=1$ for all $\al{e}{\G}$ then $\eta=\eta_t$.
\end{Remark}

Below, we collect some useful facts and begin to see the relationship between eigenvalue map twists and the weight mod $2(p-1)$. Let $E_k(\t)$ denote the weight $k$ \emph{Eisenstein series} with constant term $1$.

\begin{Proposition} \label{prop:twisted_inclusions}Let $\G$ be an $n|h$-type group with eigenvalue map $\eta$, and let $p\ge 5$ be prime with $p\nmid nh$.
 \begin{enumerate}\itemsep=0pt
 \item[$(a)$] Let $F(\tau) = E_{p-1}(h\tau)$. Then
 \begin{gather*}\hat F = \sum_{\al{e}{\G}} \qr{e}{p}F|_{p-1}W_e\in M_{p-1}(\G,\bm{1}_t)\end{gather*}
 is congruent to $\#\AL(\G)$ mod $p$.
 \item[$(b)$] For all $k$, we have the inclusions
 \begin{gather*}\tilde M_k(\G,\eta)\subseteq \tilde M_{k + p-1}(\G,\eta_t)\subseteq \tilde M_{k + 2(p-1)}(\G,\eta)\subseteq M_{k + 3(p-1)}(\G,\eta_t)\subseteq\cdots.\end{gather*}
 \item[$(c)$] Suppose $\eta\neq\eta_t$. Then for $f\in M_k^\Q(\G,\eta)$ and $f'\in M_{k'}^\Q(\G,\eta)$ we have that
 \begin{gather*}0\not\equiv f\equiv f'\pmod{p^n}\qquad\text{implies}\qquad k\equiv k'\pmod{2p^{n-1}(p-1)}.\end{gather*}
 \item[$(d)$] If $\eta\neq\eta_t$ then the weight of a $p$-adic modular form on $(\G,\eta)$ is well-defined as an element of $\hat{\mf X} = \Z_p\times\Z/(2p - 2)\Z$.
 \end{enumerate}
\end{Proposition}
\begin{proof} Since $F$ is invariant under $\Gamma_0(n|h)$, the statement that $\hat F\in M_{p-1}(\G,\bm{1}_t)$ becomes
 \begin{gather*}\hat F|W_e = \qr{e}{p}F,\end{gather*}
 which follows from multiplicativity of the Legendre symbol. We also have
 \begin{gather*}F|_{p-1}W_e(\tau) = \left(\frac{h*eh^2}{h}\right)^{\frac{p-1}{2}}E_{p-1}\big(\big(h*eh^2\big)\tau\big)\equiv \qr{e}{p} \Mod{p}\end{gather*}
 since $E_{p-1}\equiv 1\Mod{p}$ (see \cite[Section 1]{Serre72}). Thus $\hat{F}\equiv\sum\limits_{e\in\AL(\G)} 1\Mod{p}$, giving part (a). Since $p \neq 2$, this implies $1\in \tilde M_{p-1}(\G,\bm{1}_t)$, giving part (b) since $\eta\bm1_t = \eta_t$ and
 \begin{gather*}M_k(\G,\eta)\cdot M_{k'}(\G,\eta') = M_{k + k'}(\G,\eta\eta').\end{gather*}

 For part (c), we already know $k\equiv k'\pmod{p^{n-1}(p-1)}$. Assume without loss of generality that $k\le k'$. If $k\not\equiv k'\Mod{2p - 2}$ then by part (b) there exists $g\in M_{k'}^{\Q}(\G,\eta_t)$ with $g\equiv f$ $\Mod{p}$. Since $\eta\neq\eta_t$, let $\al{e}{\G}$ be a quadratic nonresidue so that
 \begin{gather*}\eta(W_e)f' = f'|_{k'}W_e\equiv g|_{k'}W_e = -\eta(W_e)g\equiv -\eta(W_e)f'\Mod{p},\end{gather*}
 which implies $f'\equiv f\equiv 0\Mod{p}$. Part (d) then follows.
\end{proof}

This motivates the following $\F_p$-spaces, which incorporate these twisted inclusions. Set
\begin{gather*}
 \tilde M(\G,\eta)^\alpha = \bigcup_{k\equiv \alpha\Mod{2p - 2}} \tilde M_k(\G,\eta) \quad \cup\quad \bigcup_{k\equiv \alpha + p - 1\Mod{2p - 2}} \tilde M_k(\G,\eta_t)
\end{gather*}
for $\alpha\in\Z/(2p - 2)\Z$. If every $\al{e}{\G}$ is a quadratic residue mod $p$, we have $\eta=\eta_t$, and $\tilde M(\G,\eta)^\alpha$ only depends on $\alpha$ mod $p-1$.

\subsection[Producing $p$-adic modular forms from Hautpmoduln]{Producing $\boldsymbol{p}$-adic modular forms from Hautpmoduln}\label{ssec:S-hauptmoduln}
Let $\G$ be a $p^\alpha n|h$-type group with eigenvalue map $\eta$ where $p\nmid nh$ is prime. Suppose $\G_\eta$ is genus zero, and its Hautpmodul $\mt$ is on $\G$ with eigenvalue map $\eta$. We show that for some $\beta$, $\mt|U_p^\beta$ is a $p$-adic modular form on $(\G',\eta')$ for some specified $n|h$-type group $\G'$ and eigenvalue map $\eta'$. We can take $\beta = 1$ whenever $\alpha\le 3$.

\begin{Lemma}\label{lemma:hauptmoduln_are_padic_mfs}Let $\G$ be a $p^\alpha n|h$-type group for $p\nmid nh$ prime, and $\eta$ be an eigenvalue map on~$\G$. Let $f$ be a weight $0$ weakly holomorphic modular form on $(\G,\eta)$ that is holomorphic away from~$\infty$. Let~$\G'$ be the $n|h$-type group with $\al{e}{\G'}$ if and only if $\al{e}{\G}$ and $p\nmid e$. Let $\beta = \max\{1,\alpha - 1\}$, and let $\eta'$ be the eigenvalue map on $\G'$ such that
\begin{gather*}\eta'(x) = \eta\big(x^{p^\beta}\big),\qquad \eta'(y') = \eta\big(y^{p^{\alpha+\beta}}\big),\qquad \eta'(W_e) = \eta(W_e)\qquad \text{for} \quad \al{e}{\G'},\end{gather*}
where $x$, $y$, $y'$ are the generators of $\G_0(p^\alpha n|h)$ and $\G_0(n|h)$ given by
\begin{gather*}x = \Matrix{1 & h \\ 0 & 1},\qquad y = \Matrix{1 & 0 \\ p^\alpha n & 1},\qquad y' = \Matrix{1 & 0 \\ n & 1}.\end{gather*}
Then, $f|U_p^\beta$ is $p$-adic modular form of weight $0$ on $(\G',\eta')$.
\end{Lemma}
\begin{proof} By Lemma~\ref{lemma:up_nh_facts}, $f|U_p^\beta$ is a weakly holomorphic modular form on $(\G,\nu)$ where $\G$ is the $pn|h$-type group $\G$ where $\al{e}{G}$ if and only if $\al{e}{\G}$ and $p\nmid e$, and $\nu$ satisfies
 \begin{gather*}\nu(x) = \eta\big(x^{p^\beta}\big),\qquad \nu(Y) = \eta\big(y^{p^{\alpha+\beta-1}}\big),\qquad \nu(W_e) = \eta(W_e),\qquad\text{where} \quad Y = \Matrix{1 & 0 \\ pn & 1}.\end{gather*}

 The remainder of the proof follows \cite[Theorem 10]{Serre72}. To show that $f|U_p^\beta$ is a $p$-adic modular form on $(\G',\eta')$, we set for $m\ge 0$
 \begin{gather*}f_m = \tr_{G_\nu\lmod \G'_{\eta'}}\big(f|U_p^\beta g^{p^m}\big) = \tr_{pnh\lmod nh}\big(f|U_p^\beta g^{p^m}\big),\end{gather*}
 where $g$ is the modular form on $\G$ given by Lemma~\ref{lemma:delta_quotient}. Since
 \begin{gather*}
 \eta'(x) = \nu(x) = \eta\big(x^{p^\beta}\big) \qquad\text{and}\qquad \eta'(y') = \nu(Y^{p}) = \eta\big(y^{p^{\alpha+\beta}}\big) \qquad \text{and} \\
 \eta'(W_e) = \nu(W_e) = \eta(W_e)\qquad \text{for} \quad \al{e}{\G'}
 \end{gather*}
 we know $f_m$ is a weakly holomorphic modular form on $(\G',\eta')$ by Lemma~\ref{lemma:trace_facts}. Lemma~\ref{lemma:hauptmodul_poles} shows that $f|U_p^\beta$ has poles only at the cusps equivalent to $0$ on $\G_0(p)$, and since $g$ has zeros at all such cusps, we know $f|U_p^\beta g^m$ is holomorphic for $m$ sufficiently large. If $a$ is the weight of $g$, the weight of $f_m$ is $ap^m$, which $p$-adically converges to $0$. Hence it suffices to show that $f_m\rightarrow f|U_p^\beta$ in the $p$-adic limit. We compute that
 \begin{gather*}
 f_m - f|U_p^\beta = \big(f_m - f|U_p^\beta g^{p^m}\big) + f|U_p^\beta\big(g^{p^m} - 1\big), \\
 v_p\big(f_m - f|U_p^\beta\big)\ge \min\big\{v_p\big(f_m - f|U_p^\beta g^{p^m}\big),v_p\big(f|U_p^\beta\big) + v_p\big(g^{p^m} - 1\big)\big\}.
 \end{gather*}
 Since $g^{p^m}\equiv 1\Mod{p^{m + 1}}$, we have $v_p\big(f|U_p^\beta\big) + v_p\big(g^{p^m} - 1\big)\ge m + 1$. Lemma~\ref{lemma:trace_facts} implies
 \begin{gather*}f_m - f|U_p^\beta g^{p^m} = p^{1 - ap^m/2}\big(f|U_p^\beta g^{p^m}|_{ap^m}W_p\big)|U_p,\end{gather*}
 and since applying $U_p$ does not decrease the power of $p$ dividing a $q$-expansion, we have
 \begin{align*}
 v_p\big(f_m - f|U_p^\beta g^{p^m}\big)&\ge 1 - \frac{ap^m}{2} + v_p\big(f|U_p^\beta|_0W_p\big) + p^mv_p(g|_aW_p)\\
 &\ge 1 + v_p\big(f|U_p^\beta|_0W_p\big) + p^m\left(v_p(g|_aW_p) - \frac{a}{2}\right).
 \end{align*}
 Lemma~\ref{lemma:delta_quotient} gives $v_p(g|_aW_p) > \frac{a}{2}$. Hence, $v_p\big(f_m - f|U_p^\beta\big)\rightarrow \infty$ as $m\rightarrow\infty$, as desired.
\end{proof}
\begin{Remark}\label{rmk:haupt_poly} If $\G$ is a $p^\alpha n|h$-type group with eigenvalue map $\eta$ such that $\G_\eta$ is genus zero, and the Hauptmodul $\mt$ on $\G_\eta$ is on $(\G,\eta)$, then Lemma~\ref{lemma:hauptmoduln_are_padic_mfs} applies. Moreover, since $\mt^r$ is on $(\G,\eta^{\sigma_r})$, the lemma also applies to powers of the Hauptmodul. In particular, polynomials in $\mt$ are $p$-adic modular forms on $\G'_{\eta'}$ after enough applications of $U_p$.
\end{Remark}

\subsection{Ordinary spaces}\label{ssec:S-hecke}
Having produced $p$-adic modular forms from Hauptmoduln on certain $n|h$-type groups, we now study the action of $U_p$. The key idea, developed by Serre on level $1$ in \cite{Serre72}, is that $U_p$ contracts mod $p$ modular forms onto a finite-dimensional space. These structural results will allow us to verify $p$-adic annihilation of certain Hauptmoduln in Section~\ref{ssec:A-ordinary}.
%this will allow us to study when these $p$-adic modular forms are annihilated under repeated application of $U_p$.

We will take $p\ge 5$ prime with $p\nmid nh$. For the Hecke operator $T_p$, we have
\begin{gather*}f|_kT_p = f|U_p + p^{k/2 - 1}f|_kA = f|U_p + p^{k-1}f(p\tau),\qquad\text{where} \quad A= \Matrix{p & 0 \\ 0 & 1},\end{gather*}
so $f|U_p\equiv f|_kT_p\Mod{p}$ for $k\ge 2$. Since $T_p$ acts on $M_k(\G_0(nh))$, we know $U_p$ acts on $\tilde M_k(\G_0(nh))$. Furthermore, let $\G$ be an $n|h$-type group and $\eta$ be an eigenvalue map. Since
\begin{gather*}A x = x^p A,\qquad A y^p = yA,\qquad A W_e = W_eA\qquad \text{for} \quad p\nmid e,\end{gather*}
Lemma~\ref{lemma:up_nh_facts} implies that $T_p\colon M_k(\G,\eta)\rightarrow M_k(\G,\eta^{\sigma_p})$. Hence $U_p\colon \tilde{M}_k(\G,\eta)\rightarrow \tilde{M}_k(\G,\eta^{\sigma_p})$, so we consider the space
\begin{gather*}\tilde M_k(\G,[\eta]_p) = \tilde M_k(\G,\eta) + \tilde M_k(\G,\eta^{\sigma_p}),\end{gather*}
which is stabilized by $U_p$. This sum is direct if and only if $\eta \neq \eta^{\sigma_p}$. We thus set
\begin{gather*}\tilde M(\G,[\eta]_p)^\alpha = \tilde M(\G,\eta)^\alpha + \tilde M(\G,\eta^{\sigma_p})^\alpha\end{gather*}
and remind the reader that $\tilde M(\G,\eta)^\alpha$ already encodes spaces with twisted eigenvalue map. We next show how $U_p$ contracts $\tilde M(\G,[\eta]_p)^\alpha$ onto a finite-dimensional space called the \textit{ordinary space}. Ordinary spaces of $p$-adic modular forms were extensively studied by Hida \cite{Hida93}. %, who proved many beautiful results describing their structure.
We describe our ordinary spaces in the language of Serre's $p$-adic modular forms.

\begin{Proposition}[ordinary decomposition]\label{prop:ordinary_decomp} Let $\G$ be an $n|h$-type group with eigenvalue map~$\eta$.
 \begin{enumerate}\itemsep=0pt
 \item[$(a)$] We can write
 \begin{gather*}\tilde M(\Gamma,[\eta]_p)^\alpha = \mf S(\G,[\eta]_p)^\alpha\oplus \mf N(\G,[\eta]_p)^\alpha\end{gather*}
 so that $U_p$ is bijective on the \textit{ordinary space} $\mf S(\G,[\eta]_p)^\alpha$ and locally nilpotent on $\mf N(\G,[\eta]_p)^\alpha$; that is, for any $\mf N(\G,[\eta]_p)^\alpha$, we have $f|U_p^n = 0$ for $n$ sufficiently large.
 \item[$(b)$] Let $4\le k\le p + 1$ be such that $k\equiv \alpha\Mod{p-1}$. If $k\equiv \alpha\Mod{2p - 2}$ then
 \begin{gather*}\mf S(\G,[\eta]_p)^\alpha\subseteq \tilde M_{k}(\G,[\eta]_p).\end{gather*}
 Otherwise,
 \begin{gather*}\mf S(\G,[\eta]_p)^\alpha\subseteq \tilde M_k(\G,[\eta_t]_p).\end{gather*}
 \end{enumerate}
\end{Proposition}

\looseness=-1 Proposition~\ref{prop:ordinary_decomp} can be interpreted to mean that repeated application of $U_p$ reduces the weight of a modular form mod $p$ to either $0$ or $p-1$. To accomplish this, we need to incorporate the twisted eigenvalue maps. More precisely, the \textit{filtration} of $f \in \tilde M(\G,[\eta]_p)^\a$ with respect to $(\G,\eta)$ is
\begin{gather*}w_{\G,\eta}(f) = \min\big\{k \colon f\in \tilde M_k(\G,[\eta]_p)\text{ or }\tilde M_k(\G,[\eta_t]_p)\big\}.\end{gather*}
When $(\G,\eta)$ is clear from context, we will simply write $w$ for $w_{\G,\eta}$. Similarly, the filtration $w_\G$ of a modular form mod $p$ with respect to $\G$ is the filtration with respect to $(\G,\bm{1})$.

To prove Proposition~\ref{prop:ordinary_decomp} we generalize the following fact from~\cite{Jochnowitz82}.
\begin{Lemma}\label{lemma:joch_filtration_facts}
 Suppose $\Gamma = \Gamma_0(N)$. Then for modular forms $f$ mod $p$ on $\G$ we have
 \begin{gather*}w_\G(f|U_p)\le p + \frac{w_\G(f) - 1}{p}.\end{gather*}
 In particular, $w_\G(f|U_p) < w_\G(f)$ if $w_\G(f) > p + 1$.
\end{Lemma}
We give a suitable modification of Lemma~\ref{lemma:joch_filtration_facts} for $(\G,\eta)$.
\begin{Lemma}\label{lemma:filtration_facts}
Let $\G$ be an $n|h$-type group with eigenvalue map $\eta$. For $f\in \tilde M(\G,[\eta]_p)^\alpha$,
 \begin{gather*}w_{\G,\eta}(f|U_p)\le p + \frac{w_{\G,\eta}(f) - 1}{p}.\end{gather*}
\end{Lemma}
\begin{proof}
 We first consider the special case $\G = \G_0(n|h)$. We proceed by induction on the finite index $[\G:\G_0(nh)]$. Suppose that for some $\G_0(nh)\le \G'\lneq \G_0(n|h)$ we have
 \begin{gather}\label{eqn:filtration_ineq}
 w_{\G',\eta|_{\G'}}(f|U_p)\le p + \frac{w_{\G',\eta|_{\G'}}(f) - 1}{p}.
 \end{gather}
 Let $T\in \G_0(n|h)/\G'$ be a representative for a nontrivial coset. Setting $\G'' = \langle \G',T\rangle$ and $\eta'' = \eta|\G''$, we will prove that (\ref{eqn:filtration_ineq}) also holds for $(\G'',\eta'')$.

 Let $f\in M_k^{(p)}(\G'',[\eta'']_p)\subseteq M_k^{(p)}(\G',[\eta|\G']_p)$ be a modular form of filtration $k$. By assumption there is some $g\in M_{k'}^{(p)}(\G',[\eta|\G']_p)$ with $g\equiv f|U_p\Mod{p}$ and of filtration
 \begin{gather*}k'\le p + \frac{k - 1}{p}.\end{gather*}
 %such that $g\equiv f|U_p\Mod{p}$.

 \noindent Let $t$ be the order of $T$ in $\G''/\G_0(nh)$, and let $\pi_m$ be the projection
 \begin{align*}
 \pi_m(g)=\frac{1}{t}\sum_{\ell = 0}^{t - 1} {\rm e}^{-2\pi {\rm i}\ell/t}g|_{k'}T^\ell
 \end{align*}
 onto the ${\rm e}^{2\pi {\rm i} m/t}$ eigenspace of $T$. Since $p\geq5$, we know $p\nmid t$. Thus, if $\varphi$, $\varphi'$ have the same weight and $\varphi\equiv\varphi'\Mod{p}$, then $\pi_m(\varphi)\equiv\pi_m(\varphi')\Mod{p}$.

 By Proposition~\ref{prop:twisted_inclusions}, multiplying by some power of $\hat F$ gives $f'\in M_k^{(p)}(\G',[\eta|\G']_p)$ with $f'\equiv g$ $\Mod{p}$. Since $\hat F$ is invariant under $\G_0(n|h)$, we have
 \begin{align*}
 \pi_m(g)\equiv\pi_m(f')\equiv\pi_m(f)\Mod{p}.
 \end{align*}
 Let $m$ be such that $\eta(T) = {\rm e}^{2\pi {\rm i} m/t}$ and set
 \begin{gather*}
 \pi=
 \begin{cases}
 \pi_m & \text{if } \pi_{pm}=\pi_{m}, \\
 \pi_m\oplus \pi_{pm} & \text{otherwise},
 \end{cases}
 \end{gather*}
 which is projection onto the span of the eigenspaces of $T$ specified by $[\eta]_p$. Thus, we have $\pi(g)\equiv \pi(f)=f\Mod{p}$. Then $\pi(g)\in M_{k'}^{(p)} (\G'',[\eta'']_p)$ with $\pi(g)\equiv f\Mod{p}$ as desired.

 Extending by Atkin--Lehner involutions $W_e$ is a similar computation. Define the projections as before, replacing $T$ with the order $t=2$ action $W_e$. The function $\hat{F}$ is not necessarily fixed by $W_e$, but rather lies in the $\big(\frac{e}{p}\big)$-eigenspace of $W_e$. Thus
 \begin{align*}
 \pi_{m+\varepsilon}(g)=\pi_m(f')\equiv\pi_m(f)\Mod{p}
 \end{align*}
 where $\varepsilon=\frac12\big(1-\big(\frac{e}{p}\big)\big)$, and we have $\pi(g)\in M_{k'}^{(p)} (\G'',[\eta'']_p)$ or $\pi(g)\in M_{k'}^{(p)} (\G'',[\eta''_t]_p)$.
\end{proof}

\begin{proof}[Proof of Proposition~\ref{prop:ordinary_decomp}] Let $k$ be as in the statement of Proposition~\ref{prop:ordinary_decomp}. Set
 \begin{gather*}\mf S(\G,[\eta]_p)^\alpha = \begin{cases}\bigcap_{n\ge 1}\im U_p^n|\tilde M_k(\G,[\eta]_p) &\text{ if }k\equiv \alpha\Mod{2p - 2}, \\ \bigcap_{n\ge 1}\im U_p^n|\tilde M_k(\G,[\eta_t]_p) &\text{ otherwise.}\end{cases}\end{gather*}
 Also set
 \begin{gather*}\mf N(\G,[\eta]_p)^\alpha = \bigcup_{n\ge 1}\ker U_p^n|\tilde M(\G,[\eta]_p)^\alpha.\end{gather*}
 Lemma~\ref{lemma:filtration_facts} shows that these spaces satisfy the conditions of the proposition.
\end{proof}

Proposition~\ref{prop:ordinary_decomp} will be fundamental for proving $p$-adic annihilation in Section~\ref{sec:annihilation}.

\section[$p$-adic annihilation]{$\boldsymbol{p}$-adic annihilation}\label{sec:annihilation}
In this section, we restrict our focus to the $n|h$-type groups appearing in monstrous moonshine, and we prove the following theorem, first mentioned in the introduction, which gives a class of Hauptmoduln from monstrous moonshine that are $p$-adically annihilated for small primes~$p$.

\begin{Theorem}\label{theorem:Up_annihilation} Let $p$ be a prime. If $\G$ is an $n|h$-type group as specified in Table~{\rm \ref{tbl:Up_annihilation}}, then the Hauptmodul on $\G_\lambda$ is $p$-adically annihilated by $U_p$.
 \begin{table}[ht] \centering\scriptsize
 \begin{tabular}{ l l|l|l|l|l }
 \multicolumn{2}{c|}{$2$} & \multicolumn{1}{c|}{$3$} & \multicolumn{1}{c|}{$5$} & \multicolumn{1}{c|}{$7$} & \multicolumn{1}{c}{$11$} \\
 \hline
 $1$ & $20|2{+}5$ & $1$ & $1$ & $1$ & $1$ \\
 $2{+}$ & $20|2{+}10$ & $2{+}$ & $2{+}$ & $2{+}$ & $3|3$ \\
 $2$ & $22{+}$ & $3{+}$ & $3{+}$ & $3|3$ & $4|2{+}$ \\
 $3{+}$ & $22{+}11$ & $3$ & $3|3$ & $4|2{+}$ & $11{+}$ \\
 $3|3$ & $24|2{+}$ & $3|3$ & $4|2{+}$ & $7{+}$ & \\
 $4{+}$ & $24{+}$ & $6{+}$ & $5{+}$ & $7$ & \\
 $4|2{+}$ & $24|2{+}3$ & $6{+}2$ & $5$ & $8|4{+}$ & \\
 $4$ & $24|6{+}$ & $6|3$ & $6{+}$ & $14{+}$ & \\
 $4|2$ & $24|4{+}6$ & $8|4{+}$ & $6|3$ & $21|3{+}$ & \\
 $5{+}$ & $24|4{+}2$ & $9{+}$ & $7{+}$ & $28|2{+}$ & \\
 $6{+}$ & $24|2{+}12$ & $9$ & $8|4{+}$ & & \\
 $6{+}3$ & $24|12$ & $12|3{+}$ & $8|4$ & & \\
 $6|3$ & $28|2{+}$ & $12|6$ & $10{+}$ & & \\
 $8{+}$ & $28{+}7$ & $15|3$ & $10{+}2$ & & \\
 $8|2{+}$ & $28|2{+}14$ & $18{+}2$ & $12|3{+}$ & & \\
 $8|4{+}$ & $32{+}$ & $18{+}$ & $15{+}$ & & \\
 $8|2$ & $32|2{+}$ & $18$ & $15|3$ & & \\
 $8$ & $36|2{+}$ & $21|3{+}$ & $16|2{+}$ & & \\
 $8|4$ & $40|4{+}$ & $24|6{+}$ & $20|2{+}$ & & \\
 $10{+}$ & $40|2{+}$ & $24|4{+}2$ & $21|3{+}$ & & \\
 $10{+}5$ & $40|2{+}20$ & $24|12$ & $24|4{+}6$ & & \\
 $11{+}$ & $44{+}$ & $27{+}$ & $25{+}$ & & \\
 $12{+}$ & $48|2{+}$ & $30|3{+}10$ & $30{+}$ & & \\
 $12|2{+}$ & $52|2{+}$ & $36{+}4$ & $30|3{+}10$ & & \\
 $12|3{+}$ & $52|2{+}26$ & $39|3{+}$ & $35{+}$ & & \\
 $12{+}3$ & $56|4{+}14$ & $42|3{+}7$ & $40|4{+}$ & & \\
 $12|2{+}6$ & $60|2{+}$ & $54{+}$ & $50{+}$ & & \\
 $12|2{+}2$ & $60|2{+}5,6,30$ & $57|3{+}$ & & & \\
 $12$ & $60|6{+}10$ & $60|6{+}10$ & & & \\
 $12|6$ & $68|2{+}$ & $84|3{+}$ & & & \\
 $16|2{+}$ & $84|2{+}$ & $93|3{+}$ & & & \\
 $16$ & $84|2{+}6,14,21$ & & & & \\
 $16{+}$ & $88|2{+}$ & & & & \\
 $20{+}$ & $104|4{+}$ & & & & \\
 $20|2{+}$ & & & & &
 \end{tabular}
 \caption{Hauptmoduln with $p$-adic annihilation.} \label{tbl:Up_annihilation}
 \end{table}
\end{Theorem}For all Hauptmoduln $\mt$ appearing in monstrous moonshine, we computed $\mathcal T|U_p^n$ for small values of $n$ and small primes $p$. These data, as well as heuristics governing the sizes of the ordinary spaces $\mf S(\Gamma,[\lambda_t]_p)^0$ for $n|h$-type groups $\Gamma$, lead us to the following conjecture (the relevancy of the ordinary space to $p$-adic annihilation is discussed in Section~\ref{ssec:A-ordinary}).
\begin{Conjecture}\label{conj:annihilation_converse} The converse to Theorem~{\rm \ref{theorem:Up_annihilation}} holds, meaning that the Hauptmoduln of Tab\-le~{\rm \ref{tbl:Up_annihilation}} are the only Hauptmoduln appearing in monstrous moonshine with $p$-adic annihilation. In particular, no Hauptmodul appearing in monstrous moonshine is $p$-adically annihilated for $p\ge 13$.
\end{Conjecture}

\begin{Remark}\label{rmk:no_annihilation_method} Given an admissible $n|h$-type group $\G$, we remark here on a method for showing~$\mt_\G$ is not $p$-adically annihilated for a given $p$. By Lemma~\ref{lemma:hauptmoduln_are_padic_mfs}, after applying $U_p$ enough times we may assume that $p\nmid nh$. Then after further applying $U_p$, we have that $\mt|U_p^\alpha\in \tilde M_{p-1}(nh)$ for some $\alpha$. Since $\tilde M_{p-1}(nh)$ is finite-dimensional, it is then straightforward to verify that $\mt$ is not $p$-adically annihilated. For small levels and primes, this method is easy to apply; for example one finds that $J$ is not $13$-adically annihilated since $J|U_{13}^2\equiv J|U_{13}\Mod{13}$. However, this method quickly starts requiring many coefficients of the Hauptmoduln and basis elements of~$M_{p-1}(nh)$, particularly when~$p$ is large.

Empirically, it appears that for $p \!=\! 2,3$ one may use a similar approach with the space~$\tilde M_4(nh)$.
\end{Remark}

\begin{Remark} \label{rmk:J_no_annihilation} It is known that $J|U_p\not\equiv 0\Mod{p}$ for any $p\ge 13$~\cite{Serre76}, and see~\cite{Elkies05} for further study of such congruences.
\end{Remark}

In Section~\ref{ssec:A-compression}, generalizing formulas of Conway and Norton \cite{Conway79}, we write down \emph{compression formulas} which relate Hauptmoduln on different groups $n|h$-type groups appearing in Moonshine. These relations will reduce the verification of Theorem~\ref{theorem:Up_annihilation} to a smaller set of groups, which will be easier to verify computationally. In Section~\ref{ssec:A-ordinary}, we utilize the theory developed in Section \ref{sec:serre} to prove annihilation for the entries in the table of Theorem~\ref{theorem:Up_annihilation} with $p\geq 5$. In Section~\ref{ssec:A-lehner}, we use separate techniques due to Lehner and verify the remaining entries, corresponding to $p=2,3$. These techniques are sufficiently explicit to give rates of $p$-adic annihilation in certain cases. Finally, we discuss in Section~\ref{ssec:A-power} a connection between $p$-adic annihilation of Hauptmoduln and the group structure of the monster group.

\subsection{Compression formulas}\label{ssec:A-compression}
Throughout this section, we use the following notation. If $\G$ is the group $\G_0(n|h)+e_1,e_2,\dots$, then we write $\G^d$ for the group $\G_0(n'|h')+e_1',e_2',\dots$ where $n' = n/\gcd(n, d)$, $h' = h/\gcd(h, d)$, and $e_1', e_2', \dots$ are the divisors of $n'/h'$ among $e_1, e_2, \dots$. This notation comes from \cite{Conway79}, where it is explained that for any element $g$ of the monster, if $\mathcal{T}_g$ is the Hauptmodul for the group $\G_\lambda$ corresponding to $g$ from moonshine, then $\mathcal{T}_{g^d}$ is the Hauptmodul for $\G^d_\lambda$. Additionally, if $\G$ is the group $\G_0(n|h)+e_1,e_2,\dots$, then we write $\gen{\G, w_e}$ for the group $\G_0(n|h)+e_1,e_2,\dots,e_1*e,e_2*e,\dots$ generated by $\G$ and $w_e$, if it exists. In this section, we will also adopt the notation $\mt_\G$ for the Hautpmodul on $\G_\lambda$.

The formula
\begin{gather}\label{eqn:conway_compression}
 p \mathcal{T}_\G | U_p = \mathcal{T}_{\G^p} - \mathcal{T}_\G \qquad \text{if} \quad w_p\in\G
\end{gather}
relating the Hauptmodul for $g^p$ with that of $g$ for $g \in \M$ appears in \cite[Section~8]{Conway79}. The following relations are of a similar form, and they allow us to connect the $p$-adic properties of Hauptmoduln on closely related groups.

\begin{Proposition}\label{prop:compression} Let $\G$ be a $p^r n|h$-type group where $p$ is prime and $p \nmid nh$. Then, whenever all of the relevant groups appear in monstrous moonshine,
 \begin{enumerate}[$(a)$]\itemsep=0pt
 \item $p \mathcal{T}_\G | U_p = \mathcal{T}_\G - \mathcal{T}_{\gen{\G, w_p}}$
 \hspace{.49 in} if $r = 1$ and $p \nmid e$ for $w_e\in\G$,
 \item $p^2 \mathcal{T}_\G | U_p^2 = \mathcal{T}_\G - \mathcal{T}_{\G^p}$
 \hspace{.61 in} if $r = 1$ and $p \nmid e$ for $w_e\in\G$,
 \item $p \mathcal{T}_\G | U_p = \mathcal{T}_{\gen{\G^p, w_{p^{r-1}}}} - \mathcal{T}_{\G^p}$
 \hspace{.15 in} if $r > 1$ and $w_{p^r}\in\G$, and
 \item $\mathcal{T}_\G | U_p = - \mathcal{T}_{\G^p} | U_p$
 \hspace{.8 in} if $r = 2$ and $w_{p^2}\in\G$.
 \end{enumerate}
\end{Proposition}
\begin{Remark} If $r = 1$, $p \nmid e$ for $w_e \in \G$, and $\gen{\G, w_p}$ appears in monstrous moonshine, then part~(b) follows from part~(a) and~(\ref{eqn:conway_compression}). However, part (b) holds even when $\gen{\G, w_p}$ is not admissible. For example $4 \mathcal{T}_{6|3} | U_2^2 = \mathcal{T}_{6|3} - \mathcal{T}_{3|3}$ even though $6|3{+}$ is not admissible. Similarly, if $r = 2$, $w_{p^2}\in\G$, and $\gen{\G^p, w_p}$ appears in monstrous moonshine, then part~(d) follows from parts~(a) and~(c), but part~(d) holds even when $\gen{\G^p, w_p}$ is not admissible.
\end{Remark}

In each case, there are only finitely many Hauptmoduln satisfying the hypotheses, and for each, we may apply Lemma~\ref{lemma:pole_bound}, use Sturm's bound, and check that sufficiently many of the coefficients are zero ($2500$ coefficients suffice in all cases). These relations reduce the number of Hauptmoduln one needs to check to show that the Hauptmoduln in Theorem~\ref{theorem:Up_annihilation} are indeed annihilated. For example, (a) implies that if $\G$ is an $pn|h$-type group with $p \nmid n$ and $\mathcal{T}_\G$ is $p$-adically annihilated, then so is $\mathcal{T}_{\gen{\G, w_p}}$. Note that many of these formulas allow us to prove that $\mathcal{T}_{\G^p}$ is $p$-adically annihilated from the fact that $\mathcal{T}_\G$ is, which explains much of the structure in the figures in Appendix~\ref{apndx:web}. In some sense, Proposition~\ref{prop:compression} suggests that $p$-adic properties of Hauptmoduln must be closely related to moonshine modules, since they tend to be preserved under power maps in the underlying group. Since we have already proved that $\mathcal{T}$ is $p$-adically annihilated in the cases where $\mathcal{T} | U_p = 0$ (see Remark~\ref{rmk:hauptmodul_killed}), we have altogether reduced the verification of Theorem~\ref{theorem:Up_annihilation} to the following much smaller check.

\begin{Corollary}\label{cly:compressed_check}In order to show that the Hauptmoduln in Theorem~{\rm \ref{theorem:Up_annihilation}} are $p$-adically annihilated, it suffices to check $p$-adic annihilation for the following smaller set shown in Table~{\rm \ref{tbl:compressed_check}}.
\end{Corollary}

\begin{table}[h]\centering\scriptsize
\begin{tabular}{ l|l|l|l|l }
 \multicolumn{1}{c|}{$2$} & \multicolumn{1}{c|}{$3$} & \multicolumn{1}{c|}{$5$} & \multicolumn{1}{c|}{$7$} & \multicolumn{1}{c}{$11$} \\
 \hline
 $6{+}3$ & $6{+}2$ & $4|2{+}$ & $2{+}$ & $3|3$ \\
 $6|3$ & $24|4{+}2$ & $6{+}$ & $3|3$ & $4|2{+}$ \\
 $10{+}5$ & & $6|3$ & $4|2{+}$ & $11{+}$ \\
 $22{+}11$ & & $7{+}$ & $8|4{+}$ & \\
 & & $8|4{+}$ & & \\
 & & $8|4$ & & \\
 & & $12|3{+}$ & & \\
 & & $16|2{+}$ & & \\
 & & $21|3{+}$ & & \\
 & & $24|4{+}6$ & & \\
 & & $30|3{+}10$ & &
\end{tabular}
\caption{Hauptmoduln for which it suffices to prove Theorem~\ref{theorem:Up_annihilation}.}\label{tbl:compressed_check}
\end{table}

One has a significant amount of freedom in choosing these representatives -- we have chosen those most conducive to performing computations when $p = 2,3$. When $p\ge 5$ we will prove Theorem~\ref{theorem:Up_annihilation} directly in Section~\ref{ssec:A-ordinary} rather than using the reduction given here.

\subsection{Annihilation via ordinary spaces}\label{ssec:A-ordinary}
In this section, we will prove Theorem~\ref{theorem:Up_annihilation} for $p\ge 5$ using the theory of $p$-adic modular forms developed in Section~\ref{sec:serre}. The key observation relating the theory of $p$-adic modular forms to $p$-adic annihilation is the following easy consequence of Proposition~\ref{prop:ordinary_decomp}.
\begin{Lemma}\label{lemma:small_dim_annihilation}
 Suppose that $\mf S(\G,[\eta]_p)^0\subseteq \F_p$. If $f$ is a weight $0$ $p$-adic modular form on $(\G,\eta)$, we have $f|U_p^n\rightarrow c$ in the $p$-adic limit, where $c$ is the constant term of $f$.
\end{Lemma}
When $f = \mt|U_p$ for $\mt$ a Hauptmodul, then $c = 0$. By Lemma~\ref{lemma:hauptmoduln_are_padic_mfs}, $f$ is a weight $0$ $p$-adic modular form on $(\G,\eta)$ for some $\G$, $\eta$. Hence Lemma~\ref{lemma:small_dim_annihilation} applies.

Using the \texttt{mfslashexpansion} and \texttt{mfatkin} functions in Pari \cite{Pari/GP}, one can compute the actions of
\begin{gather*}x = \Matrix{1 & h \\ 0 & 1}, \qquad y = \Matrix{1 & 0 \\ n & 1},\end{gather*}
and all Atkin--Lehner involutions on $M_{p-1}(\G,[\lambda_t]_p)$. Then using elementary linear algebra, it is easy to find a basis for $\mf S(\G,[\lambda]_p)^0\subseteq \tilde M_{p-1}(\G,[\lambda_t]_p)$. We performed this computation for various $n|h$ groups $\G$ appearing in monstrous moonshine with small values of $n$ with $p\nmid n$ (specifically, $n\le 24$ for $p = 5$, $n\le 11$ for $p = 7$, and $n\le 7$ for $p = 11$).

The $\G$ for which $\mf S(\G,[\lambda]_p)^0\subseteq \F_p$ are given in Table~\ref{tbl:annihilated_Gammas}. Applying Lemma~\ref{lemma:hauptmoduln_are_padic_mfs} to the Hauptmoduln of Table~\ref{tbl:Up_annihilation} then proves Theorem~\ref{theorem:Up_annihilation} for $p\ge 5$. We note that every group from Table~\ref{tbl:annihilated_Gammas} corresponds to at least one group from Theorem~\ref{theorem:Up_annihilation}.

\begin{table}[h]\centering\scriptsize
 \begin{tabular}{c|c c c c c c c c c c c c c }
 $p = 5$ & $1$ & $2{+}$ & $3{+}$ & $3|3$ & $4|2{+}$ & $6{+}$ & $6|3$ & $7{+}$ & $8|4{+}$ & $8|4$ & $12|3{+}$ & $16|2{+}$ & $24|4{+}6$ \\ \hline
 $p = 7$ & $1$ & $2{+}$ & $3|3$ & $4|2{+}$ & $8|4{+}$ \\ \hline
 $p = 11$ & $1$ & $3|3$ & $4|2{+}$
 \end{tabular}
\caption{$\G$ such that $\mf S(\G,[\lambda]_p)^0\subset\F_p$.}\label{tbl:annihilated_Gammas}
\end{table}

\begin{Remark} Another method for computing $\mf S(\G,[\lambda]_p)$ is as follows. By drawing fundamental domains for $\Gamma_{\lambda_t}$, one can compute the dimension of the space $M_{p-1}(\Gamma_{\lambda_t})$. One then uses the $q$-expansions of the modular forms $\mt^r|U_p^m$ to compute a basis for $\tilde M_{p-1}(\G,[\lambda_t]_p)$. Computing $\mf S(\G,[\lambda_t]_p)^0$ then amounts to linear algebra. We carried out this procedure in Sage~\cite{Sage}, giving an alternate verification of the results in Table~\ref{tbl:annihilated_Gammas}. The Sturm's bound calculations for this method require checking $3500$ coefficients.
\end{Remark}

It is worth noting that these methods also apply to Hauptmoduln not appearing in monstrous moonshine. For example, if $\G = \G_0(2|2)$ and $\mt$ is the Hauptmodul on $\G_\lambda$ then Lemma~\ref{lemma:hauptmoduln_are_padic_mfs} gives that $\mt|U_5$ is a $5$-adic modular form on $(\G,\lambda)$. Using the method above, one can compute that $\mf S(\G,[\lambda]_p)^0 = 0$, so that $\mt$ is $5$-adically annihilated.

Although the methods developed here do not directly give rates of annihilation, the following observation held in all cases from Table~\ref{tbl:annihilated_Gammas}, when checked with $10000$ coefficients.

\begin{Remark}\label{rmk:rates}Let $\mathcal{T}$ be the Hauptmodul on some $\Gamma_\lambda$ appearing in Table~\ref{tbl:Up_annihilation} for $p\geq 5$, $\Gamma'$ from Table \ref{tbl:annihilated_Gammas} as given by Lemma~\ref{lemma:hauptmoduln_are_padic_mfs}, and $\lambda'$ the corresponding eigenvalue map specified by Conway--Norton. Let $m$ be the smallest integer such that $U_p^m\tilde M_{p-1}(\G',[\lambda']_p) = \mf S(\G',[\lambda']_p)^0$. Numerically, we observe $v_p\big(\mt|U_p^{\ell+m+1}\big)\geq v_p\big(\mt|U_p^{\ell}\big)+1$, bounding the rate of annihilation from below by $1/(m+1)$. Moreover this choice of $m$ appears to be tight. We pose the question of whether these observations continue to hold to in general.
\end{Remark}

\subsection[Additional $p$-adic annihilation]{Additional $\boldsymbol{p}$-adic annihilation}\label{ssec:A-lehner}
The ordinary spaces of Section~\ref{ssec:S-hecke} and the annihilation verified in Section \ref{ssec:A-ordinary} were restricted to $p\geq 5$. In this section, we use different techniques to explicitly verify $p$-adic annihilation for the six groups appearing in Corollary~\ref{cly:compressed_check} for the primes $p=2,3$.

\begin{Proposition}\label{prop:cooking}We have
\begin{align*}
 v_p\big(\mt|U_p^n\big)\geq\lfloor n\alpha \rfloor,
\end{align*}
when $\mt$ is the normalized Hauptmodul on $\G_\lambda$, with parameters given in Table~{\rm \ref{tbl:cooking}}.
\begin{table}[ht] \centering
 \begin{tabular}{c|c|c|c|c|c|c}
 $\G$ & $6{+}2$ & $6{+}3$ & $10{+}5$ & $22{+}11$ & $6|3$ & $24|4{+}2$ \\ \hline
 $p$ & $3$ & $2$ & $2$ & $2$ & $2$ & $3$ \\
 $\alpha$ & $3/2$ & $1$ & $3/2$ & $1/2$ & $3/2$ & $1/2$
 \end{tabular}
 \caption{Rates of annihilation.} \label{tbl:cooking}
\end{table}
\end{Proposition}
\begin{proof}Throughout, we use $\mathfrak{T}$ to denote the normalized Hauptmodul $\mt$ plus some constant, and $\eta(\tau) = q^{1/24}\prod\limits_{n\ge 1}(1 - q^n)$ is the \textit{Dedekind eta function}. The constant will be specified by writing $\mathfrak{T}$ as an eta-quotient given in \cite{Conway79} and recorded in Table \ref{tbl:eta_quotients}.
\begin{table}[ht] \centering
 \begin{tabular}{c|c|c|c|c|c|c}
 $\G$ & $6{+}2$ & $6{+}3$ & $10{+}5$ & $22{+}11$ & $6|3$ & $24|4{+}2$ \\ \hline
 \tsep{4pt}$\mathfrak{T}$ & $\frac{\eta(\tau)^4\eta(2\tau)^4}{\eta(3\tau)^4\eta(6\tau)^4}$ & $\frac{\eta(\tau)^6\eta(3\tau)^6}{\eta(2\tau)^6\eta(6\tau)^6}$ & $\frac{\eta(\tau)^4\eta(5\tau)^4}{\eta(2\tau)^4\eta(10\tau)^4}$ & $\frac{\eta(\tau)^2\eta(11\tau)^2}{\eta(2\tau)^2\eta(22\tau)^2}$ & $\frac{\eta(3\tau)^8}{\eta(6\tau)^8}$ & $\frac{\eta(4\tau)\eta(8\tau)}{\eta(12\tau)\eta(24\tau)}$
 \end{tabular}
 \caption{Eta quotients for unnormalized Hauptmoduln.} \label{tbl:eta_quotients}
\end{table}
We follow the methods of \cite{Lehner49a,Lehner49b} for the $j$-function, and prove our six cases in parallel.

We first compute expansions of $\mathfrak{T}$ at each cusp, and write $\mathfrak{T}|U_p$ as a~rational function in a~Hauptmodul by subtracting off the principal part. Below, we write out the calculation for $p=3$ and the group $6{+}2$, and record only the formulas for the rest. These formulas may also be verified via a Sturm's bound coefficient check ($600$ suffice in all cases). Let
\begin{gather*}
S_\lambda=\begin{pmatrix}1 & \lambda \\ 0 & p \end{pmatrix} \qquad \text{and} \qquad F_2=\begin{pmatrix} 0 & -1 \\ 2 & 0 \end{pmatrix}.
\end{gather*}
For $W_2$ a suitable Atkin--Lehner matrix on $\G_0(6)$, we have
\begin{align*}
 3(\mathfrak{T}|U_3)\left(-1/2\t\right)&=\sum_{\lambda=0}^2\mathfrak{T}(S_\lambda F_2)
 =\mathfrak{T}\left(-1/6\t\right)+\sum_{\lambda=1}^2\mathfrak{T}(S_\lambda F_2\tau)\\
 &=\mathfrak{T}\left(-1/6\t\right)+\sum_{\mu=1}^2\mathfrak{T}(W_2 T_\mu\tau) =\mathfrak{T}\left(-1/6\t\right)+3(\mathfrak{T}|U_3)(\tau)-\mathfrak{T}(\tau/3),
\end{align*}
and upon substituting $\tau\mapsto 3\tau$ we obtain the cusp expansion
\begin{align*}
 3(\mathfrak{T}|U_3)\left(-1/6\t\right)=-q^{-1}+O(1).
\end{align*}
We next subtract off this principal part with $-\mathfrak{T}(\tau)$. Sending $\tau\mapsto-1/6\t$ and applying the $\eta$ functional equation to the quotients in Table~\ref{tbl:eta_quotients}, we find
\begin{gather*}
3(\mathcal{T}|U_3)(\tau)=-3^4\mathfrak{T}(\tau)^{-1}+c
\end{gather*}
for some constant $c$. One can check that $\mathfrak{T}|U_p$ is bounded at the other cusps; see Lemma~\ref{lemma:hauptmodul_poles}. For appropriate constants $c$, similar computations show
\begin{gather}\label{eq:hauptmodul_functional}
 p(\mathcal{T}|U_p)(\tau)=-p^e\mathfrak{T}(\tau)^{-1}+c,
\end{gather}
where $e$ is as in Table~\ref{tbl:e_powers}.
\begin{table}[ht] \centering
 \begin{tabular}{c|c|c|c|c|c|c}
 $\G$ & $6{+}2$ & $6{+}3$ & $10{+}5$ & $22{+}11$ & $6|3$ & $24|4{+}2$ \\ \hline
 $e$ & $4$ & $6$ & $4$ & $2$ & $4$ & $1$
 \end{tabular}
 \caption{Parameters for functional equations.} \label{tbl:e_powers}
\end{table}

These relations imply $v_p(\mathcal{T}|U_p)\geq e-1$. Set $Z=\mathfrak{T}^{-1}$ and $W=p^e Z(\tau/p)$. Expanding $Z$ at other cusps and applying the principal part analysis as above, we find
\begin{align*}
 p(p^e Z|U_p)(\tau)=f(p^e Z).
\end{align*}
where $f$ is a polynomial with $f(\mathfrak{T}(\tau))=\mathfrak{T}(p\tau)+O(q)$ (different for each group). Write
\begin{align}
 \label{eq:hauptmodul_polynomial}
 (Z|U_p)(\t)=p^{\alpha}\sum_{j=1}^p b_j Z^j
\end{align}
where the coefficients $b_j$ are listed in Table~\ref{tbl:bj_coeffs}.
\begin{table}[ht]\centering
\begin{tabular}{c|c|c|c|c|c|c}
 & $6{+}2$ & $6{+}3$ & $10{+}5$ & $22{+}11$ & $6|3$ & $24|4{+}2$ \\ \hline
 $b_1$ & $2\cdot3^{1/2}$ & $3$ & $2^{1/2}$ & $2^{1/2}$ & 0 & 0\\
 $b_2$ & $4\cdot 3^{5/2}$ & $2^4$ & $2^{3/2}$ & $2^{1/2}$ & $2^{3/2}$ & 0\\
 $b_3$ & $3^{11/2}$ & & & & & $3^{1/2}$
\end{tabular}
\caption{Polyomial coefficients.}\label{tbl:bj_coeffs}
\end{table}

Following Lehner \cite[equation~(2.2)]{Lehner49b} gives
\begin{gather*}
 W^p+\sum_{j=1}^p (-1)^j p_j W^{p-j}=0 , \qquad \text{where} \quad (-1)^{j+1}p_j =p^{e+\alpha+1}\sum_{m=j}^p b_m Z^{m-j+1}.
\end{gather*}
This equation has roots $W(\tau+\lambda)$ for $\lambda\in\{0,\ldots,p-1\}$. If $S_\ell$ denotes the sum of the $\ell$-th power of these roots, then we have $p^{-e\ell-1}S_\ell=Z^\ell|U_p$. We show that $p^{\alpha(\ell-1)}(Z^\ell|U_p)\in p^\alpha Z\cdot \Z[p^\alpha Z,p^\alpha]$ for all $\ell\geq 1$. Equivalently, we check that $S_\ell\in p^{(e-\alpha)\ell+1+\alpha}p^\alpha Z\cdot \Z[p^\alpha Z,p^\alpha]$.

Lehner uses Newton sums to set up an induction, relating $S_\ell$ to $S_j$ for $j<\ell$. A similar computation works for our cases. The coefficients in Table~\ref{tbl:bj_coeffs} imply the base case $\ell=1$. By Newton sums we have
\begin{gather*}
 S_\ell=\sum_{j=1}^\ell (-1)^{j+1} p_j S_{\ell-j},
\end{gather*}
where $b_j,p_h=0$ for $j,h\geq p+1$ and $S_0=\ell$. For $\ell\leq p$ we rewrite the Newton sum as
\begin{gather*}
 S_\ell =\sum_{j=1}^{\ell-1} (-1)^{j+1}p_j S_{\ell-j}-(-1)^\ell \ell p_\ell.
\end{gather*}
 By construction we have that $p_j\in p^{e+1+\alpha+v_p(b_j)-\alpha}p^{\alpha}Z\cdot \Z[p^{\alpha}Z,p^{\alpha}]$ and applying the inductive hypothesis gives $p_j S_{\ell-j}\in p^{e+1+\alpha+v_p(b_j)-\alpha+(e-\alpha)(\ell-j)+1+\alpha} p^{\alpha}Z\cdot \Z[p^{\alpha}Z]$. We have the inequality $e+1+\alpha+v_p(b_j)-\alpha\geq (e-\alpha) j$, which one checks explicitly using the values of $b_j$. In particular, $p_j S_{\ell-j}\in p^{(e-\alpha)\ell+1+\alpha} p^\alpha Z\cdot \Z[p^\alpha Z,p^\alpha]$. We also check that $\ell p_\ell\in p^{(e-\alpha)\ell+1+\alpha} p^{\alpha}Z\cdot\Z[p^{\alpha}Z,p^{\alpha}]$ from the explicit values of $b_j$. Thus $S_\ell\in p^{(e-\alpha)\ell+1+\alpha}p^\alpha Z\cdot \Z[p^\alpha Z,p^\alpha]$ for $\ell\leq p$.
 When $l\geq p+1$ we have
 \begin{gather*}
 S_\ell =\sum_{j=1}^{\ell-1} (-1)^{j+1}p_j S_{\ell-j}
 \end{gather*}
 and we check that $p_j S_{\ell-j}\in p^{(e-\alpha)\ell+1+\alpha} p^\alpha Z\cdot\Z[p^\alpha Z,p^\alpha]$ as above, which completes the induction.
 Thus, we have shown that
 \begin{gather*}
 \big(p^{\alpha(\ell-1)}Z^\ell|U_p\big)\in p^\alpha Z\cdot \Z[p^\alpha Z,p^\alpha]
 \end{gather*}
 so that if $g\in Z\cdot \Z[p^\alpha Z,p^\alpha]$, then $g|U_p\in p^\alpha Z\cdot \Z[p^\alpha Z,p^\alpha]$. Note that $\mt|U_p=-p^{e-1}Z\in Z\cdot \Z[p^\alpha Z,p^\alpha]$. Repeatedly applying the result of the induction thus gives the rates of annihilation claimed in Proposition~\ref{prop:cooking}.
 \end{proof}

\begin{Remark}These explicit techniques also apply when $p\geq 5$, as long as $\mt|U_p$ is modular on a genus zero group. For example, if $\mt$ is the Hauptmodul for $2{+}$, then $\mt|U_5$ is modular on $10{+}2$, which is genus zero. There are many examples for $p\geq 5$ where this does not hold, e.g., if $\mt$ the Hauptmodul for $2{+}$ then $\mt|U_7$ is modular on $14{+}2$, which is genus one. For this example, the arguments of this section can be modified to obtain explicit lower bounds on the annihilation rate, e.g., by working with bivariate polynomials in appropriate meromorphic modular forms instead of single-variate polynomials in the Hauptmodul, since the latter do not exist for the genus one curve corresponding to $14{+}2$. In general, however, we rely on the theory of Section~\ref{sec:serre} to prove $p$-adic annihilation as in Section~\ref{ssec:A-ordinary}.
\end{Remark}

\subsection{Preservation of annihilation under power maps}\label{ssec:A-power}
We discuss the relationship between Hauptmoduln $p$-adic annihilation and power maps in the monster. We have already seen hints of this in Section \ref{ssec:A-compression}. Unlike previous sections, this section will not be used elsewhere, except as motivation for Section \ref{sec:moonshine}.

\begin{Definition}For a fixed Hauptmodul $\mt_g$ and an integer $d$, we say that $p$-adic annihilation is \emph{preserved under the $d$-th power map} if $\mt_g$ $p$-adically annihilated implies $\mt_{g^d}$ is $p$-adically annihilated.
\end{Definition}

These relationships are depicted in Appendix \ref{apndx:web} for primes $p\leq 11$, which we expect are the only primes with annihilation. Note that if $d$ is relatively prime to the order of $g$, then $\mt_{g^d} = \mt_g$, so $p$-adic annihilation is preserved under the $d$-th power map. We will therefore restrict our attention to those $d$ that divide the order of $g$. We first give conceptual explanations for this preservation of $p$-adic annihilation, when it holds. We then characterize exactly when $p$-adic annihilation is not preserved, and offer a notion of $p$-adic annihilation that seems to always be preserved under power maps. The average numerical rates of annihilation from Appendix~\ref{apndx:table} often do not decrease under power maps~-- we remark upon this briefly.

In certain situations, the compression formulas show that Hauptmodul $p$-adic annihilation is preserved under power maps, e.g. powering from the group $\ell+$ to level $1$, when $\ell$ is prime. When $p\geq 5$, we may also explain via ordinary spaces.

\begin{Example}When $p=5$, the groups $50+$, $10+2$, $10+$, and $2+$ all have the same ordinary space $\mathbb{F}_p=\mf{S}(\G_0(2)+,\bm{1})^0\subseteq \tilde{M}_{p-1}(\G_0(2)+,\bm{1}_t)$, and their Hauptmoduln share the same annihilation behavior. This explains preservation of $p$-adic annihilation under $p$-th power maps. The compression formulas of Proposition \ref{prop:compression} capture these same relations.
\end{Example}

Furthermore, in light of Lemma~\ref{lemma:small_dim_annihilation} and Section~\ref{ssec:A-ordinary}, we seek inclusions of the form $\mf S(\G^d,[\lambda]_p)^0$ $\hookrightarrow\mf S(\G,[\lambda]_p)^0$ for $\G$ being $n|h$-type with $p\nmid nh$. This would explain $d$-th power map preservation of $p$-adic annihilation from $\G_\lambda$ to $\G^d_\lambda$, and similarly for groups with the same ordinary spaces. The following proposition accomplishes this in certain situations.
\begin{Proposition} Let $p\geq5$ be prime, $(d,nh)=1$, $\G$ be an $n|h$-type group with eigenvalue map $\eta$, and $\G'$ be a $dn|h$-type group with eigenvalue map $\eta'$ such that
 \begin{gather*}\AL(\G) = \AL(\G'),\qquad (d,e) = 1 \qquad \text{for all} \quad \al{e}{\G},\qquad\text{and}\qquad \eta|_{\G'} = \eta'.\end{gather*}
 \begin{enumerate}\itemsep=0pt
 \item[$(a)$] There is an inclusion $M_k(\G,\eta)\subseteq M_k(\G',\eta')$.
 \item[$(b)$] For any $\alpha\in \Z/(2p - 2)\Z$, there is an inclusion $\mf S(\G,[\eta]_p)^\alpha\subseteq \mf S(\G',[\eta']_p)^\alpha$.
 \end{enumerate}
\end{Proposition}
\begin{proof} Part (b) follows from part (a) using the description of $\mf S$ given in the proof of Proposition~\ref{prop:ordinary_decomp}. Part (a) follows from the isomorphic inclusion $\G/\G_0(dnh)\hookrightarrow \G'/\G_0(nh)$ from Section~\ref{ssec:T-traces}.
\end{proof}

\begin{Remark}Inclusions of the form $M_k(\G,\eta)\subseteq M_k(\G',\eta')$ may have further implications for preserving average rates of $p$-adic annihilation via Remark~\ref{rmk:rates}.
\end{Remark}

We now discuss when power maps do not preserve Hauptmodul $p$-adic annihilation.

\begin{Example}\label{ex:h_power}If $\mt$ is the Hauptmodul on $\G_\lambda$ for $\G$ an $n|h$-type group, then the Hauptmodul on $\G_\lambda^{d}$ need not be annihilated when $d\divides h$, which describe all the grey boxes in the figures of Appendix \ref{apndx:web}. For example, the Hauptmodul for $6|3$ is $5$-adically annihilated, but the Hauptmodul for $2$ is not.
\end{Example}

\begin{Example}\label{ex:faux_infinity}If $\mt$ is the Hauptmodul on $\G_\lambda$ and $\mt|U_p=0$, power maps from $\G_\lambda$ may not preserve $p$-adic annihilation. For example, if $\mt$ is the Hauptmodul on $12$, $\mt|U_p=0$ while the Hauptmodul on $6$ is not $2$-adically annihilated.
\end{Example}

The classes of groups from Examples \ref{ex:h_power} and \ref{ex:faux_infinity} exhaust the situations for which power maps do not preserve Hauptmodul $p$-adic annihilation. If one further excludes the groups $\G$ with $\mt_\G|U_p = 0$, then rates of $p$-adic annihilation are preserved under power maps. These examples motivate a notion of strong $p$-adic annihilation for genus zero groups. This strong notion of annihilation will not be referenced outside of this section.

\begin{Definition}Let $\G$ be a genus zero group with Hauptmodul $\mt$. We say that $\G$ has \emph{strong $p$-adic annihilation} if every polynomial $f(\mt)$ in the Hauptmodul with no constant term is $p$-adically annihilated.
\end{Definition}

In contrast with $p$-adic annihilation of Hauptmoduln, strong $p$-adic annihilation numerically appears to be always preserved under power maps. Furthermore, our numerical data indicate that rates of strong $p$-adic annihilation are also non-decreasing under power maps. We discussed polynomials in the Hauptmodul in Remark \ref{rmk:haupt_poly}, and one finds that $\Gamma_\eta$ has strong $p$-adic annihilation if $\mf{S}_{p-1}(\G,\bm{1}_t)^0=\mathbb{F}_p$ and $\mf{S}_{p-1}(\G,[(\lambda^{\sigma_r})_t]_p)^0=0$ otherwise.

 We can check that strong $p$-adic annihilation is now preserved for the groups of Example \ref{ex:h_power}. Indeed, let $\mt$ be the Hauptmodul on $\G_\lambda$ for $\G$ an $n|h$-type group, and $\mt'$ be the Hauptmodul on $\G_\lambda^{d}$ for~$d\divides h$. We recall the~$V_m$ operator, given by $f|V_m=f\circ A$ where we set $A=\left(\begin{smallmatrix} m & 0 \\ 0 & 1 \end{smallmatrix}\right)$ as in Section~\ref{ssec:S-hecke}. It acts on Fourier expansions by $\big(\sum a(n)q^n\big)|V_m = \sum a(n)q^{mn}$. From \cite[Section~6]{Conway79}, we know that $\mt'|V_{d}+c=\mt^{d}$ for some constant $c$. In particular, if $\mt'$ is not $p$-adically annihilated, $\mt'|V_{d}$ is a polynomial in $\mt$ with no constant term and is also not $p$-adically annihilated. Thus, if $\G_\lambda^{d}$ does not have strong $p$-adic annihilation, neither does $\G_\lambda$. This reflects the following inclusions.

\begin{Proposition} \label{prop:bar_inclusion} Let $p\geq5$ and $\G$ be an $n|h$-type group with eigenvalue map $\eta$, and let $\G'$ be a $nd|hd$ type group with eigenvalue map $\eta'$ such that
 \begin{gather*}\AL(\G) = \AL(\G'), \qquad \eta'(x')=\eta(x),\qquad \eta'(y')=\eta(y),\qquad \text{and}\\
 \eta'(W_e)=\eta(W_e) \qquad \text{for all}\quad \al{e}{\G} \end{gather*}
 for $x$, $y$ and $x'$, $y'$ the usual generators of $\G_0(n|h)/\G_0(nh)$ and $\G_0(nd|hd)/\G_0\big(nhd^2\big)$.
 \begin{enumerate}\itemsep=0pt
 \item[$(a)$] There is an injection given by the operator $V_{d}\colon M_k(\G,\eta)\hookrightarrow M_k(\G',\eta')$.
 \item[$(b)$] For any $\alpha\in \Z/(2p - 2)\Z$, this gives an injection $V_{d}\colon \mf S(\G,[\eta]_p)^\alpha\hookrightarrow \mf S(\G',[\eta']_p)^\alpha.$
 \end{enumerate}
\end{Proposition}

We observe empirically that strong $p$-adic annihilation is also preserved for the groups from Example~\ref{ex:faux_infinity}, e.g., $\mt^2$ is not $2$-adically annihilated when $\mt$ is the Hauptmodul on $12$. %In

Further aspects of this connection between $p$-adic properties of modular forms and the associated conjugacy classes of the monster group will be discussed in Section \ref{sec:moonshine}.

\section{Moonshine}\label{sec:moonshine}
In this section, we will investigate groups with $p$-adic moonshine. Recall from the introduction that a \textit{moonshine module} for a finite group $G$ is a graded $G$-module $V = \bigoplus_{n=-1}^\infty V_n$ such that
\begin{enumerate}\itemsep=0pt
\item[$(i)$] For each $g \in G$, the McKay--Thompson series
 \begin{gather*}\mathcal{T}_g(\t) = \sum_{n=-1}^\infty \tr(g \big| V_n) q^n\end{gather*}
 is the Hauptmodul of an order $\ord(g)$ conjugacy class of the monster.

\item[$(ii)$] For any $g \in G$ and any $n \in \Z$, if $\mathcal{T}_g$ is the Hauptmodul for $\G$, then $\mathcal{T}_{g^n}$ is the Hauptmodul for $\G^n$.
\end{enumerate}
If for some prime $p$, $V$ also satisfies the following property, then we call $V$ a \textit{$p$-adic moonshine module}.
\begin{enumerate}\itemsep=0pt
 \item[(iii)] For each irreducible character $\chi$ of $\G$, if $m_\chi(n)$ is the multiplicity of $\chi$ in the character of~$V_n$, then the multiplicity generating function
 \begin{gather*}
 \mathcal{M}_\chi(\t) = \sum_{n=-1}^\infty m_\chi(n) q^n
 \end{gather*}
 associated to $\chi$ is $p$-adically annihilated.
\end{enumerate}

Throughout this section, we assume that Conjecture \ref{conj:annihilation_converse} holds. That is, we assume that Theorem \ref{theorem:Up_annihilation} exactly characterizes which of the $171$ Hauptmoduln appearing in monstrous moonshine are $p$-adically annihilated for each prime $p$.

We begin in Section \ref{ssec:M-order} by stating basic facts about groups with $p$-adic moonshine, including the fact that for each prime $p$, only finitely many such groups exist. In Section \ref{ssec:M-examples}, we illustrate $p$-adic moonshine with several examples of groups having $p$-adic moonshine modules. In Section~\ref{ssec:M-monster}, we find a surprising class of subgroups of the monster having $p$-adic moonshine in a slightly more general sense.

\subsection{Basic facts}\label{ssec:M-order}
We begin by presenting an alternative formulation of $p$-adic moonshine that will allow us to make use of the results in Section \ref{sec:annihilation}.

\begin{Lemma}\label{lemma:multiplicities}Let $G$ be a finite group with a moonshine module $V = \bigoplus_{n=-1}^\infty V_n$. For any prime~$p$, $V$ is a $p$-adic moonshine module if and only if the McKay--Thompson series $\mathcal{T}_g$ is $p$-adically annihilated for each $g \in G$.
\end{Lemma}
\begin{proof}
 By the Schur orthogonality relations, we have
 \begin{gather*}
 \mathcal{T}_g = \sum_{\chi \in \Irrep(G)} \chi(g) \mathcal{M}_\chi \qquad \text{for each} \quad g \in G
 \end{gather*}
 and
 \begin{gather*}
 \mathcal{M}_\chi = \frac{1}{|G|} \sum_{g \in G} \conj{\chi(g)} \mathcal{T}_g \qquad \text{for each}\quad \chi \in \Irrep(G).
 \end{gather*}
 If each $\mathcal{M}_\chi$ is $p$-adically annihilated, then for all $n \in \N$, there exists $N \in \N$ such that $\mathcal{M}_\chi | U_p^N \equiv 0$ $\Mod{p^n}$. Since each $\chi(g)$ is an algebraic integer, it follows that $\mathcal{T}_g | U_p^N \equiv 0 \Mod{p^n}$, so each~$\mathcal{T}_g$ is $p$-adically annihilated. Similarly, if each $\mathcal{T}_g$ is $p$-adically annihilated, then each $|G| \mathcal{M}_\chi$, and hence each $\mathcal{M}_\chi$, is also $p$-adically annihilated.
\end{proof}

As an immediate consequence, a group $G$ has $p$-adic moonshine if and only if we can assign to each element $g \in G$ a $p$-adically annihilated Hauptmodul $\mathcal{T}_g$ in a way that agrees with power maps in $G$ and so that the multiplicities defined by
\begin{gather*}\mathcal{M}_\chi = \frac{1}{|G|} \sum_{g \in G} \conj{\chi(g)} \mathcal{T}_g\end{gather*}
have positive integral coefficients for all irreducible characters $\chi$. In particular, each McKay--Thompson series $\mathcal{T}_g$ must be a Hauptmodul for one of the groups listed in Theorem~\ref{theorem:Up_annihilation}. In fact, since we require that the assignment of group elements to Hauptmoduln agree with power maps, we can restrict our attention to only those congruence groups $\G$ with the property that $\G^n$ is $p$-adically annihilated for every $n \in \Z$. By Theorem \ref{theorem:Up_annihilation}, it follows that every McKay--Thompson series $\mathcal{T}_g$ must be a Hauptmodul for one of the groups listed in Table \ref{tbl:candidate_mts}.

\begin{table}[ht]\centering\scriptsize
\begin{tabular}{ l l|l|l|l|l }
 \multicolumn{2}{c|}{$2$} & \multicolumn{1}{c|}{$3$} & \multicolumn{1}{c|}{$5$} & \multicolumn{1}{c|}{$7$} & \multicolumn{1}{c}{$11$} \\
 \hline
 $1$ & $12|3{+}$ & $1$ & $1$ & $1$ & $1$ \\
 $2{+}$ & $12{+}3$ & $2{+}$ & $2{+}$ & $2{+}$ & $3|3$ \\
 $2$ & $12|6$ & $3{+}$ & $3{+}$ & $3|3$ & $11{+}$ \\
 $3{+}$ & $16|2{+}$ & $3$ & $3|3$ & $4|2{+}$ & \\
 $3|3$ & $16$ & $3|3$ & $4|2{+}$ & $7{+}$ & \\
 $4{+}$ & $16{+}$ & $6{+}$ & $5{+}$ & $7$ & \\
 $4|2{+}$ & $20{+}$ & $6{+}2$ & $5$ & $8|4{+}$ & \\
 $4$ & $20|2{+}$ & $9{+}$ & $6{+}$ & $14{+}$ & \\
 $4|2$ & $20|2{+}5$ & $9$ & $7{+}$ & $21|3{+}$ & \\
 $5{+}$ & $22{+}$ & $18{+}2$ & $8|4{+}$ & $28|2{+}$ & \\
 $6{+}$ & $22{+}11$ & $18{+}$ & $10{+}$ & & \\
 $6{+}3$ & $24|2{+}$ & $27{+}$ & $10{+}2$ & & \\
 $6|3$ & $24{+}$ & $54{+}$ & $15{+}$ & & \\
 $8{+}$ & $24|2{+}3$ & & $15|3$ & & \\
 $8|2{+}$ & $24|6{+}$ & & $20|2{+}$ & & \\
 $8|4{+}$ & $24|12$ & & $21|3{+}$ & & \\
 $8|2$ & $32{+}$ & & $25{+}$ & & \\
 $8$ & $32|2{+}$ & & $30{+}$ & & \\
 $8|4$ & $40|4{+}$ & & $35{+}$ & & \\
 $10{+}$ & $40|2{+}$ & & $40|4{+}$ & & \\
 $10{+}5$ & $44{+}$ & & $50{+}$ & & \\
 $11{+}$ & $48|2{+}$ & & & & \\
 $12{+}$ & $88|2{+}$ & & & & \\
 $12|2{+}$ & & & & &
\end{tabular}
\caption{Candidate McKay--Thompson series for groups with $p$-adic moonshine.}\label{tbl:candidate_mts}
\end{table}

As a first application of this simpler description of $p$-adic moonshine, the following proposition bounds the powers of primes dividing the orders of groups with $p$-adic moonshine for each prime~$p$.

\begin{Proposition}\label{prop:order_bounds}Let $p$ be a prime and $G$ be a group with $p$-adic moonshine. Assuming that Conjecture~{\rm \ref{conj:annihilation_converse}} holds, the following table gives for each prime~$q$ a value $r$ such that every group~$G$ with $p$-adic moonshine satisfies $v_q(|G|) \le r$.
$$\begin{array}{l|c c c c c c}
 & p = 2 & p = 3 & p = 5 & p = 7 & p = 11 & p > 11 \\
 \hline
 q = 2 & 46 & 12 & 15 & 15 & 0 & 0 \\
 q = 3 & 8 & 21 & 8 & 3 & 3 & 0 \\
 q = 5 & 3 & 0 & 9 & 0 & 0 & 0 \\
 q = 7 & 0 & 0 & 2 & 6 & 0 & 0 \\
 q = 11 & 2 & 0 & 0 & 0 & 2 & 0 \\
 q > 11 & 0 & 0 & 0 & 0 & 0 & 0
\end{array}$$
\end{Proposition}
\begin{proof} First, note that if $q \divides |G|$, then by Cauchy's theorem, $G$ contains an element of order $q$, so there must be some Hauptmodul corresponding to a $q|h$-type group in Table~\ref{tbl:candidate_mts}. This proves all of the cases with $r = 0$ above.

For the remaining cases with $r > 0$, we follow the method given in \cite{DeHority18}. Let $J, \mathcal{T}_1, \dots, \mathcal{T}_n$ be the distinct Hauptmoduln given in Table~\ref{tbl:candidate_mts} as candidate McKay--Thompson series for groups with $p$-adic moonshine. For each $i \in \{1, \dots, n\}$, let $a_i$ be the number of elements of $G$ whose McKay--Thompson series is $\mathcal{T}_i$. Then, the multiplicity of the trivial character $\e$ is
 \begin{gather*}\mathcal{M}_\e
 = \frac{1}{|G|} \sum_{g \in G} \mathcal{T}_g
 = \frac{1}{|G|} \prn{J + a_1 \mathcal{T}_1 + \dots + a_n \mathcal{T}_n}.\end{gather*}
 In particular, since $\mathcal{M}_\e$ must have integral coefficients, it follows that
 \begin{gather*}J + a_1 \mathcal{T}_1 + \dots + a_n \mathcal{T}_n \equiv 0 \Mod{|G|}.\end{gather*}
 Therefore, if we choose $r$ large enough so that there are no coefficients $a_1, \dots, a_n$ such that
 \begin{gather*}J + a_1 \mathcal{T}_1 + \dots + a_n \mathcal{T}_n \equiv 0 \Mod{q^{r+1}},\end{gather*}
 then we have shown that $v_q(|G|) \le r$. As in \cite{DeHority18}, this computation was carried out using Sage~\cite{Sage} by computing the kernel of the matrix of coefficients of~$J, \mathcal{T}_1, \dots, \mathcal{T}_n$.
\end{proof}

\subsection[Examples of groups with $p$-adic moonshine]{Examples of groups with $\boldsymbol{p}$-adic moonshine}\label{ssec:M-examples}
We begin by illustrating the process of showing a group has $p$-adic moonshine using the group~$A_5$ and the prime $p = 5$. The character table for~$A_5$ is given below:
\begin{table}[h!]
 \centering
 $\begin{array}{c|ccccc}
 A_5 & 1 & (1\;2)(3\;4) & (1\;2\;3) & (1\;2\;3\;4\;5) & (1\;3\;4\;5\;2) \\
 \hline
 \e & 1 & 1 & 1 & 1 & 1 \\
 \chi_1 & 4 & 0 & 1 & -1 & -1 \\
 \chi_2 & 5 & 1 & -1 & 0 & 0 \\
 \chi_3 & 3 & -1 & 0 & \frac{1+\sqrt{5}}{2} & \frac{1-\sqrt{5}}{2} \\
 \chi_4 & 3 & -1 & 0 & \frac{1-\sqrt{5}}{2} & \frac{1+\sqrt{5}}{2}
 \end{array}$
 \caption{Character Table for $A_5$.}\label{tab:A5_character_table}
\end{table}

The square of any element of the fourth conjugacy class listed in the character table is in the fifth conjugacy class, so any moonshine module must have the same McKay--Thompson series for elements of those two conjugacy classes. The only other non-trivial power relations come from the fact that $g^{\ord(g)} = 1$ for any $g$, so a possible assignment of $5$-adically annihilated McKay--Thompson series that agrees with power maps in $A_5$ is given by assigning $\mathcal{T}_1$ to the element of the conjugacy class of $1$, $\mathcal{T}_{2{+}}$ to the elements of the conjugacy class of $(1\;2)(3\;4)$, $\mathcal{T}_{3|3}$ to the elements of the conjugacy class of $(1\;2\;3)$, and $\mathcal{T}_5$ to the elements of the conjugacy classes of $(1\;2\;3\;4\;5)$ and $(1\;3\;4\;5\;2)$. The associated multiplicities are then given by
\begin{gather*}
 \mathcal{M}_\e = q^{-1} + 4378 q + 382380 q^2 + 14714988 q^3 + 340105628 q^4 + O\big(q^5\big), \\
 \mathcal{M}_{\chi_1} = 13122 q + 1432996 q^2 + 57620010 q^3 + 1349723748 q^4 + O\big(q^5\big), \\
 \mathcal{M}_{\chi_2} = 17500 q + 1815128 q^2 + 72334998 q^3 + 1689829376 q^4 + O\big(q^5\big), \\
 \mathcal{M}_{\chi_3} = 8753 q + 1050626 q^2 + 42904992 q^3 + 1009618126 q^4 + O\big(q^5\big), \qquad \text{and} \\
 \mathcal{M}_{\chi_4} = 8753 q + 1050626 q^2 + 42904992 q^3 + 1009618126 q^4 + O\big(q^5\big).
\end{gather*}
In order to show that this gives a valid moonshine module, we must show that these multiplicities are both positive and integral. For positivity, we may use inequality~(4.16) in \cite{DeHority18}. Indeed, this inequality holds for $n = 2$, and hence for all $n \ge 2$ since the left-hand side is monotonically increasing. Since the first coefficients of each multiplicity generating function are positive, this implies that all of them must be. For integrality, we may simply note that each multiplicity generating function is on $\G_0(90)$, so we may use Sturm's bound to reduce the computation to checking only the coefficients up to $q^{216}$. Using Sage \cite{Sage}, it turns out to indeed be the case that all of the coefficients are integers. Thus, $A_5$ has $5$-adic moonshine with the McKay--Thompson series given above.

In fact, it turns out that $p$-adic moonshine is not such a rare phenomenon among groups with small orders. Indeed, using Sage \cite{Sage}, we have computed that for every prime $p$, every group $G$ of order at most $25$ for which there is some assignment of $p$-adically annihilated Hauptmoduln to elements of $G$ obeying power maps has $p$-adic moonshine. In fact, out of all $45252$ such feasible assignments, only $11$ do not give rise to a $p$-adic moonshine module, and all $11$ exceptions are for $p = 2$ and the group $G = \Zmod{2} \times \prn{\Zmod{3}}^2$.

In certain special cases, these computations become somewhat simpler. For example, consider the case of a non-trivial group $G$ in which we assign every non-identity element the same Hauptmodul~$\mathcal{T}$. In particular, this means that every non-identity element of~$G$ must have order~$q$ for some prime~$q$. We will characterize exactly when $G$ has a $p$-adic moonshine module under this assignment.

Under such an assignment, the multiplicity of the trivial character $\e$ is given by
\begin{gather*}\mathcal{M}_\e
= \frac{1}{|G|} \sum_{g \in G} \mathcal{T}_g
= \frac{1}{|G|} (J - \mathcal{T}) + \mathcal{T},\end{gather*}
and the multiplicity of any non-trivial character $\chi$ is given by
\begin{gather*}\mathcal{M}_\chi
= \frac{1}{|G|} \sum_{g \in G} \conj{\chi(g)} \mathcal{T}_g
= \frac{1}{|G|} \prn{\conj{\chi(e)} J + \sum_{g \in G - \{e\}} \conj{\chi(g)} \mathcal{T}}
= \frac{\dim \chi}{|G|} (J - \mathcal{T}).\end{gather*}
These multiplicities are both integral if and only if $|G| \divides (J - \mathcal{T})$, and for checking positivity, we may once again use inequality (4.16) in \cite{DeHority18}. After a computation in Sage \cite{Sage}, we have the following result.

\begin{Proposition}\label{prop:pgroup_moonshine}
 Let $p$ and $q$ be primes, $G$ be a group of exponent $q$, and $\mathcal{T}$ be a Hauptmodul for one of the order $q$ conjugacy class of the monster. Then, assuming that Conjecture~{\rm \ref{conj:annihilation_converse}} holds, $G$ has a $p$-adic moonshine module in which the McKay--Thompson series for each non-identity element is $\mathcal{T}$ if and only if $p$ and $\mathcal{T}$ appear in the following table and $|G| \le q^r$ where $r$ is the corresponding entry in the third row.
$$\begin{array}{c|l l l l l l l l l l l}
 p & 2 & 2 & 2 & 3 & 3 & 3 & 5 & 5 & 7 & 7 & 11 \\ \hline
 \mathcal{T} & 2{+} & 2 & 3|3 & 3{+} & 3 & 3|3 & 5{+} & 5 & 7{+} & 7 & 11{+} \\ \hline
 r & 12 & 13 & 1 & 6 & 9 & 3 & 3 & 5 & 2 & 4 & 2
 \end{array}$$
\end{Proposition}

\subsection{Centralizers in the monster}\label{ssec:M-monster}
Other examples of $p$-adic moonshine come from subgroups of the monster meeting only those conjugacy classes whose Hauptmoduln are $p$-adically annihilated, though Proposition~\ref{prop:pgroup_moonshine} shows that not every group with $p$-adic moonshine is of this form. In this section, we will exhibit a~surprising class of subgroups of the monster having $p$-adic moonshine in a slightly more general sense.

Specifically, we say that a modular function $f$ is \textit{weakly $p$-adically annihilated} if $f | U_p^n \equiv 0$ $\Mod{p}$ for some $n \ge 0$, and that a moonshine module is a \textit{weakly $p$-adic moonshine module} if each McKay--Thompson series is weakly $p$-adically annihilated. For $p \ge 5$, every McKay--Thompson series in monstrous moonshine that is weakly $p$-adically annihilated is also $p$-adically annihilated, so weakly $p$-adic moonshine and $p$-adic moonshine coincide. For $p \in \{2,3\}$, however, these notions diverge. In addition to those in Table~\ref{tbl:Up_annihilation}, the McKay--Thompson series given in Table~\ref{tbl:weak_annihilation} are weakly $p$-adically annihilated. In each case, this may be verified with a finite check using Sturm's bound (with at most $3500$ coefficients).

\begin{table}[h!]\centering\scriptsize
\begin{tabular}{ l l l|l l l }
 \multicolumn{3}{c|}{$2$} & \multicolumn{3}{c}{$3$} \\
 \hline
 $3$ & $15{+}$ & $30{+}$ & $2$ & $13{+}$ & $28{+}$ \\
 $5$ & $15{+}5$ & $30{+}3,5,15$ & $4{+}$ & $14{+}$ & $30{+}6,10,15$ \\
 $6{+}6$ & $15|3$ & $30{+}5,6,30$ & $4|2{+}$ & $14{+}7$ & $30{+}$ \\
 $6{+}2$ & $17{+}$ & $30|3{+}10$ & $4|2$ & $15{+}$ & $31{+}$ \\
 $6$ & $18{+}$ & $33{+}$ & $5{+}$ & $15{+}5$ & $34{+}$ \\
 $7{+}$ & $18{+}9$ & $34{+}$ & $5$ & $15{+}15$ & $36{+}$ \\
 $7$ & $19{+}$ & $35{+}$ & $6{+}6$ & $16|2{+}$ & $36|2{+}$ \\
 $9{+}$ & $20{+}4$ & $36{+}$ & $6{+}3$ & $17{+}$ & $39{+}$ \\
 $10{+}2$ & $20{+}20$ & $38{+}$ & $6$ & $18{+}9$ & $40|4{+}$ \\
 $10{+}10$ & $21{+}$ & $41{+}$ & $7{+}$ & $18{+}18$ & $40|2{+}$ \\
 $10$ & $21{+}21$ & $42{+}$ & $8|2{+}$ & $19{+}$ & $42{+}$ \\
 $12{+}4$ & $24{+}8$ & $42{+}6,14,21$ & $8|4$ & $20|2{+}$ & $45{+}$ \\
 $12{+}12$ & $24{+}24$ & $51{+}$ & $10{+}$ & $20|2{+}5$ & $48|2{+}$ \\
 $13{+}$ & $26{+}$ & $56{+}$ & $10{+}10$ & $20|2{+}10$ & $51{+}$ \\
 $13$ & $26{+}26$ & $60{+}$ & $12{+}$ & $21{+}$ & $60|2{+}$ \\
 $14{+}$ & $28{+}$ & $66{+}$ & $12{+}4$ & $24|2{+}$ & $60|2{+}5,6,30$ \\
 $14{+}7$ & $29{+}$ & $70{+}$ & $12|2{+}$ & $24|4{+}6$ & $62{+}$ \\
 $14{+}14$ & & & $12|2{+}6$ & $26{+}$ & $78{+}$ \\
 & & & $12|2{+}2$ & $28|2{+}$ & $84|2{+}$
\end{tabular}
\caption{Additional weakly $p$-adically annihilated McKay--Thompson series.}\label{tbl:weak_annihilation}
\end{table}

Using GAP \cite{GAP}, we have found that for each $p \in {2, 3, 5, 7, 11}$, the centralizer of a $pA$-pure elementary abelian subgroup of the monster of order $p^2$ has weakly $p$-adic moonshine given by restricting the monster module. The ATLAS names of these groups, which were found using~\cite{Atlas, Wilson88}, are given in Table \ref{tbl:monster_subgroups}. In fact, each group in Table \ref{tbl:monster_subgroups} intersects only those conjugacy classes whose Hauptmoduln $\mathcal{T}$ satisfy $\mathcal{T} | U_p \equiv 0 \Mod{p}$, which is somewhat stronger than weakly $p$-adic annihilation.

\begin{table}[ht]\centering
\begin{tabular}{ c|c|c|c|c|c }
 $p$ & $2$ & $3$ & $5$ & $7$ & $11$ \\ \hline
 $C\big(pA^2\big)$ & $2^2 \cdot {}^2\!E_6(2)$ & $3^2 \times O_8^+(3)$ & $5^2 \times U_3(5)$ & $7^2 \times L_2(7)$ & $11^2$ \tsep{2pt}\bsep{2pt}\\ \hline
\tsep{2pt} $\#C\big(pA^2\big)$ & $2^{38}{\cdot}3^9{\cdot}5^2{\cdot}7^2{\cdot}11{\cdot}13{\cdot}17{\cdot}19$ & $2^{12} {\cdot} 3^{14} {\cdot} 5^2 {\cdot} 7 {\cdot} 13$ & $2^4 {\cdot} 3^2 {\cdot} 5^5 {\cdot} 7$ & $2^3 {\cdot} 3 {\cdot} 7^3$ & $11^2$
\end{tabular}
\caption{Subgroups of the monster with weakly $p$-adic moonshine.}\label{tbl:monster_subgroups}
\end{table}

In light of this, it is natural to ask whether there are other natural subgroups of the monster having $p$-adic or weakly $p$-adic moonshine and whether there is an explanation intrinsic to the monster for the existence of weakly $p$-adic moonshine for these subgroups. More generally, we pose the question of whether the results of this paper extend to other known cases of moonshine, such as Conway moonshine~\cite{Duncan15B}, umbral moonshine \cite{Duncan15}, and Thompson moonshine \cite{Harvey15}. Do analogues of the groups in Table~\ref{tbl:monster_subgroups} exist for these other groups?

\appendix

\section{Table of annihilation}\label{apndx:table}
The following table gives the precise congruences that numerically appear to be satisfied by each Hauptmodul. The notation $a_1, \dots, a_m \to b_1, \dots, b_n$ means the sequence beginning $a_1, \dots, a_m$, and then each subsequent term is given by adding $b_{k \Mod{n}}$ from $k=1$ to $\infty$. For example, $0, 1 \to 0, 3$ is the sequence $0,1,1,4,4,7,7,10,10,\dots$. The entry under $p$ for the group $\G$ indicates the sequence $a_1, a_2, a_3, \dots$ such that $a_n$ is the highest power of $p$ dividing $\mt_\G | U_p^n$. If no such cyclic pattern is clear, then we simply list the first few terms of the sequence in parentheses.

\begin{scriptsize}
\begin{longtable}{l|l|l|l|l|l|l}
 Class & Group & $p = 2$ & $p = 3$ & $p = 5$ & $p = 7$ & $p = 11$ \\ \hline
 1A & $1$ & $11 \to 3$ & $5 \to 2$ & $2 \to 1$ & $1 \to 1$ & $1 \to 1$ \\
 2A & $2{+}$ & $11 \to 3$ & $(3,5,9,9,11)$ & $(1,2,3,5,5)$ & $1 \to 1$ & $0 \to 0$ \\
 2B & $2$ & $11 \to 3$ & $1 \to 0$ & $0 \to 0$ & $0 \to 0$ & $0 \to 0$ \\
 3A & $3{+}$ & $5 \to 1$ & $5 \to 2$ & $(1,2,3,5,5)$ & $0 \to 0$ & $0 \to 0$ \\
 3B & $3$ & $2 \to 0$ & $5 \to 2$ & $0 \to 0$ & $0 \to 0$ & $0 \to 0$ \\
 3C & $3|3$ & $3 \to 3,0$ & $\infty$ & $0 \to 1,0$ & $1 \to 1$ & $0 \to 1,0$ \\
 4A & $4{+}$ & $11 \to 3$ & $1 \to 0$ & $0 \to 0$ & $0 \to 0$ & $0 \to 0$ \\
 4B & $4|2{+}$ & $\infty$ & $1 \to 0$ & $1 \to 1$ & $1 \to 1$ & $0 \to 1,0$ \\
 4C & $4$ & $\infty$ & $0 \to 0$ & $0 \to 0$ & $0 \to 0$ & $0 \to 0$ \\
 4D & $4|2$ & $\infty$ & $1 \to 0$ & $0 \to 0$ & $0 \to 0$ & $0 \to 0$ \\
 5A & $5{+}$ & $3 \to 2,1,3,0$ & $1 \to 0$ & $2 \to 1$ & $0 \to 0$ & $0 \to 0$ \\
 5B & $5$ & $1 \to 0$ & $1 \to 0$ & $2 \to 1$ & $0 \to 0$ & $0 \to 0$ \\
 6A & $6{+}$ & $5 \to 1$ & $(3,5,9,9,11)$ & $1 \to 1$ & $0 \to 0$ & $0 \to 0$ \\
 6B & $6{+}6$ & $2 \to 0$ & $1 \to 0$ & $0 \to 0$ & $0 \to 0$ & $0 \to 0$ \\
 6C & $6{+}3$ & $5 \to 1$ & $1 \to 0$ & $0 \to 0$ & $0 \to 0$ & $0 \to 0$ \\
 6D & $6{+}2$ & $2 \to 0$ & $(3,5,9,9,11)$ & $0 \to 0$ & $0 \to 0$ & $0 \to 0$ \\
 6E & $6$ & $2 \to 0$ & $1 \to 0$ & $0 \to 0$ & $0 \to 0$ & $0 \to 0$ \\
 6F & $6|3$ & $3 \to 3,0$ & $\infty$ & $0 \to 1,0$ & $0 \to 0$ & $0 \to 0$ \\
 7A & $7{+}$ & $2 \to 0$ & $1 \to 0$ & $1 \to 1$ & $1 \to 1$ & $0 \to 0$ \\
 7B & $7$ & $2 \to 0$ & $0 \to 0$ & $0 \to 0$ & $1 \to 1$ & $0 \to 0$ \\
 8A & $8{+}$ & $7 \to 3$ & $0 \to 0$ & $0 \to 0$ & $0 \to 0$ & $0 \to 0$ \\
 8B & $8|2{+}$ & $\infty$ & $1 \to 0$ & $0 \to 0$ & $0 \to 0$ & $0 \to 0$ \\
 8C & $8|4{+}$ & $\infty$ & $0 \to 1,0$ & $1 \to 1$ & $0 \to 1,0$ & $0 \to 0$ \\
 8D & $8|2$ & $\infty$ & $0 \to 0$ & $0 \to 0$ & $0 \to 0$ & $0 \to 0$ \\
 8E & $8$ & $\infty$ & $0 \to 0$ & $0 \to 0$ & $0 \to 0$ & $0 \to 0$ \\
 8F & $8|4$ & $\infty$ & $1 \to 0$ & $0 \to 1,0$ & $0 \to 0$ & $0 \to 0$ \\
 9A & $9{+}$ & $1 \to 0$ & $5 \to 2$ & $0 \to 0$ & $0 \to 0$ & $0 \to 0$ \\
 9B & $9$ & $0 \to 0$ & $\infty$ & $0 \to 0$ & $0 \to 0$ & $0 \to 0$ \\
 10A & $10{+}$ & $3 \to 2,1,3,0$ & $1 \to 0$ & $(1,2,3,5,5)$ & $0 \to 0$ & $0 \to 0$ \\
 10B & $10{+}5$ & $3 \to 2,1,3,0$ & $0 \to 0$ & $0 \to 0$ & $0 \to 0$ & $0 \to 0$ \\
 10C & $10{+}2$ & $1 \to 0$ & $0 \to 0$ & $(1,2,3,5,5)$ & $0 \to 0$ & $0 \to 0$ \\
 10D & $10{+}10$ & $1 \to 0$ & $1 \to 0$ & $0 \to 0$ & $0 \to 0$ & $0 \to 0$ \\
 10E & $10$ & $1 \to 0$ & $0 \to 0$ & $0 \to 0$ & $0 \to 0$ & $0 \to 0$ \\
 11A & $11{+}$ & $1 \to 1,0$ & $0 \to 0$ & $0 \to 0$ & $0 \to 0$ & $1 \to 1$ \\
 12A & $12{+}$ & $5 \to 1$ & $1 \to 0$ & $0 \to 0$ & $0 \to 0$ & $0 \to 0$ \\
 12B & $12{+}4$ & $2 \to 0$ & $1 \to 0$ & $0 \to 0$ & $0 \to 0$ & $0 \to 0$ \\
 12C & $12|2{+}$ & $\infty$ & $1 \to 0$ & $0 \to 0$ & $0 \to 0$ & $0 \to 0$ \\
 12D & $12|3{+}$ & $3 \to 3,0$ & $\infty$ & $0 \to 1,0$ & $0 \to 0$ & $0 \to 0$ \\
 12E & $12{+}3$ & $\infty$ & $0 \to 0$ & $0 \to 0$ & $0 \to 0$ & $0 \to 0$ \\
 12F & $12|2{+}6$ & $\infty$ & $1 \to 0$ & $0 \to 0$ & $0 \to 0$ & $0 \to 0$ \\
 12G & $12|2{+}2$ & $\infty$ & $1 \to 0$ & $0 \to 0$ & $0 \to 0$ & $0 \to 0$ \\
 12H & $12{+}12$ & $2 \to 0$ & $0 \to 0$ & $0 \to 0$ & $0 \to 0$ & $0 \to 0$ \\
 12I & $12$ & $\infty$ & $0 \to 0$ & $0 \to 0$ & $0 \to 0$ & $0 \to 0$ \\
 12J & $12|6$ & $\infty$ & $\infty$ & $0 \to 0$ & $0 \to 0$ & $0 \to 0$ \\
 13A & $13{+}$ & $2 \to 0$ & $1 \to 0$ & $0 \to 0$ & $0 \to 0$ & $0 \to 0$ \\
 13B & $13$ & $1 \to 0$ & $0 \to 0$ & $0 \to 0$ & $0 \to 0$ & $0 \to 0$ \\
 14A & $14{+}$ & $2 \to 0$ & $1 \to 0$ & $0 \to 0$ & $1 \to 1$ & $0 \to 0$ \\
 14B & $14{+}7$ & $2 \to 0$ & $1 \to 0$ & $0 \to 0$ & $0 \to 0$ & $0 \to 0$ \\
 14C & $14{+}14$ & $2 \to 0$ & $0 \to 0$ & $0 \to 0$ & $0 \to 0$ & $0 \to 0$ \\
 15A & $15{+}$ & $1 \to 0$ & $1 \to 0$ & $(1,2,3,5,5)$ & $0 \to 0$ & $0 \to 0$ \\
 15B & $15{+}5$ & $1 \to 0$ & $1 \to 0$ & $0 \to 0$ & $0 \to 0$ & $0 \to 0$ \\
 15C & $15{+}15$ & $0 \to 0$ & $1 \to 0$ & $0 \to 0$ & $0 \to 0$ & $0 \to 0$ \\
 15D & $15|3$ & $1 \to 0$ & $\infty$ & $0 \to 1,0$ & $0 \to 0$ & $0 \to 0$ \\
 16A & $16|2{+}$ & $\infty$ & $0,1 \to 0$ & $0 \to 1,0$ & $0 \to 0$ & $0 \to 0$ \\
 16B & $16$ & $\infty$ & $0 \to 0$ & $0 \to 0$ & $0 \to 0$ & $0 \to 0$ \\
 16C & $16{+}$ & $4,6 \to 3$ & $0 \to 0$ & $0 \to 0$ & $0 \to 0$ & $0 \to 0$ \\
 17A & $17{+}$ & $1 \to 0$ & $0,1 \to 0$ & $0 \to 0$ & $0 \to 0$ & $0 \to 0$ \\
 18A & $18{+}2$ & $0 \to 0$ & $\infty$ & $0 \to 0$ & $0 \to 0$ & $0 \to 0$ \\
 18B & $18{+}$ & $1 \to 0$ & $(3,5,9,9,11)$ & $0 \to 0$ & $0 \to 0$ & $0 \to 0$ \\
 18C & $18{+}9$ & $1 \to 0$ & $1 \to 0$ & $0 \to 0$ & $0 \to 0$ & $0 \to 0$ \\
 18D & $18$ & $0 \to 0$ & $\infty$ & $0 \to 0$ & $0 \to 0$ & $0 \to 0$ \\
 18E & $18{+}18$ & $0 \to 0$ & $1 \to 0$ & $0 \to 0$ & $0 \to 0$ & $0 \to 0$ \\
 19A & $19{+}$ & $1,2 \to 0$ & $1 \to 0$ & $0 \to 0$ & $0 \to 0$ & $0 \to 0$ \\
 20A & $20{+}$ & $3 \to 2,1,3,0$ & $0 \to 0$ & $0 \to 0$ & $0 \to 0$ & $0 \to 0$ \\
 20B & $20|2{+}$ & $\infty$ & $1 \to 0$ & $1 \to 1$ & $0 \to 0$ & $0 \to 0$ \\
 20C & $20{+}4$ & $1 \to 0$ & $0 \to 0$ & $0 \to 0$ & $0 \to 0$ & $0 \to 0$ \\
 20D & $20|2{+}5$ & $\infty$ & $0,1 \to 0$ & $0 \to 0$ & $0 \to 0$ & $0 \to 0$ \\
 20E & $20|2{+}10$ & $\infty$ & $1 \to 0$ & $0 \to 0$ & $0 \to 0$ & $0 \to 0$ \\
 20F & $20{+}20$ & $1 \to 0$ & $0 \to 0$ & $0 \to 0$ & $0 \to 0$ & $0 \to 0$ \\
 21A & $21{+}$ & $1 \to 0$ & $1 \to 0$ & $0 \to 0$ & $0 \to 0$ & $0 \to 0$ \\
 21B & $21{+}3$ & $0 \to 0$ & $0 \to 0$ & $0 \to 0$ & $0 \to 0$ & $0 \to 0$ \\
 21C & $21|3{+}$ & $0 \to 0$ & $\infty$ & $0 \to 1,0$ & $1 \to 1$ & $0 \to 0$ \\
 21D & $21{+}21$ & $1 \to 0$ & $0 \to 0$ & $0 \to 0$ & $0 \to 0$ & $0 \to 0$ \\
 22A & $22{+}$ & $1 \to 1,0$ & $0 \to 0$ & $0 \to 0$ & $0 \to 0$ & $0 \to 0$ \\
 22B & $22{+}11$ & $1 \to 1,0$ & $0 \to 0$ & $0 \to 0$ & $0 \to 0$ & $0 \to 0$ \\
 23AB & $23{+}$ & $0 \to 0$ & $0 \to 0$ & $0 \to 0$ & $0 \to 0$ & $0 \to 0$ \\
 24A & $24|2{+}$ & $\infty$ & $1 \to 0$ & $0 \to 0$ & $0 \to 0$ & $0 \to 0$ \\
 24B & $24{+}$ & $3 \to 1$ & $0 \to 0$ & $0 \to 0$ & $0 \to 0$ & $0 \to 0$ \\
 24C & $24{+}8$ & $1 \to 0$ & $0 \to 0$ & $0 \to 0$ & $0 \to 0$ & $0 \to 0$ \\
 24D & $24|2{+}3$ & $\infty$ & $0 \to 0$ & $0 \to 0$ & $0 \to 0$ & $0 \to 0$ \\
 24E & $24|6{+}$ & $\infty$ & $\infty$ & $0 \to 0$ & $0 \to 0$ & $0 \to 0$ \\
 24F & $24|4{+}6$ & $\infty$ & $1 \to 0$ & $(0,0,1,2,2)$ & $0 \to 0$ & $0 \to 0$ \\
 24G & $24|4{+}2$ & $\infty$ & $0 \to 1,0$ & $0 \to 0$ & $0 \to 0$ & $0 \to 0$ \\
 24H & $24|2{+}12$ & $\infty$ & $0 \to 0$ & $0 \to 0$ & $0 \to 0$ & $0 \to 0$ \\
 24I & $24{+}24$ & $1 \to 0$ & $0 \to 0$ & $0 \to 0$ & $0 \to 0$ & $0 \to 0$ \\
 24J & $24|12$ & $\infty$ & $\infty$ & $0 \to 0$ & $0 \to 0$ & $0 \to 0$ \\
 25A & $25{+}$ & $0 \to 0$ & $0 \to 0$ & $2 \to 1$ & $0 \to 0$ & $0 \to 0$ \\
 26A & $26{+}$ & $2 \to 0$ & $0,1 \to 0$ & $0 \to 0$ & $0 \to 0$ & $0 \to 0$ \\
 26B & $26{+}26$ & $1 \to 0$ & $0 \to 0$ & $0 \to 0$ & $0 \to 0$ & $0 \to 0$ \\
 27AB & $27{+}$ & $0 \to 0$ & $2 \to 2$ & $0 \to 0$ & $0 \to 0$ & $0 \to 0$ \\
 28A & $28|2{+}$ & $\infty$ & $0,1 \to 0$ & $0 \to 0$ & $1 \to 1$ & $0 \to 0$ \\
 28B & $28{+}$ & $2 \to 0$ & $1 \to 0$ & $0 \to 0$ & $0 \to 0$ & $0 \to 0$ \\
 28C & $28{+}7$ & $\infty$ & $0 \to 0$ & $0 \to 0$ & $0 \to 0$ & $0 \to 0$ \\
 28D & $28|2{+}14$ & $\infty$ & $0 \to 0$ & $0 \to 0$ & $0 \to 0$ & $0 \to 0$ \\
 29A & $29{+}$ & $1 \to 0$ & $0 \to 0$ & $0 \to 0$ & $0 \to 0$ & $0 \to 0$ \\
 30A & $30{+}6,10,15$ & $0 \to 0$ & $1 \to 0$ & $0 \to 0$ & $0 \to 0$ & $0 \to 0$ \\
 30B & $30{+}$ & $1 \to 0$ & $1 \to 0$ & $1 \to 1$ & $0 \to 0$ & $0 \to 0$ \\
 30C & $30{+}3,5,15$ & $1 \to 0$ & $0 \to 0$ & $0 \to 0$ & $0 \to 0$ & $0 \to 0$ \\
 30D & $30{+}5,6,30$ & $1 \to 0$ & $0 \to 0$ & $0 \to 0$ & $0 \to 0$ & $0 \to 0$ \\
 30E & $30|3{+}10$ & $1 \to 0$ & $\infty$ & $0 \to 1,0$ & $0 \to 0$ & $0 \to 0$ \\
 30F & $30{+}2,15,30$ & $0 \to 0$ & $0 \to 0$ & $0 \to 0$ & $0 \to 0$ & $0 \to 0$ \\
 30G & $30{+}15$ & $0 \to 0$ & $0 \to 0$ & $0 \to 0$ & $0 \to 0$ & $0 \to 0$ \\
 31AB & $31{+}$ & $0 \to 0$ & $1 \to 0$ & $0 \to 0$ & $0 \to 0$ & $0 \to 0$ \\
 32A & $32{+}$ & $2,3,5 \to 3$ & $0 \to 0$ & $0 \to 0$ & $0 \to 0$ & $0 \to 0$ \\
 32B & $32|2{+}$ & $\infty$ & $0 \to 0$ & $0 \to 0$ & $0 \to 0$ & $0 \to 0$ \\
 33A & $33{+}11$ & $0 \to 0$ & $0 \to 0$ & $0 \to 0$ & $0 \to 0$ & $0 \to 0$ \\
 33B & $33{+}$ & $1 \to 0$ & $0 \to 0$ & $0 \to 0$ & $0 \to 0$ & $0 \to 0$ \\
 34A & $34{+}$ & $1 \to 0$ & $0,1 \to 0$ & $0 \to 0$ & $0 \to 0$ & $0 \to 0$ \\
 35A & $35{+}$ & $1 \to 0$ & $0 \to 0$ & $1 \to 1$ & $0 \to 0$ & $0 \to 0$ \\
 35B & $35{+}35$ & $0 \to 0$ & $0 \to 0$ & $0 \to 0$ & $0 \to 0$ & $0 \to 0$ \\
 36A & $36{+}$ & $1 \to 0$ & $1 \to 0$ & $0 \to 0$ & $0 \to 0$ & $0 \to 0$ \\
 36B & $36{+}4$ & $0 \to 0$ & $\infty$ & $0 \to 0$ & $0 \to 0$ & $0 \to 0$ \\
 36C & $36|2{+}$ & $\infty$ & $1 \to 0$ & $0 \to 0$ & $0 \to 0$ & $0 \to 0$ \\
 36D & $36{+}36$ & $0 \to 0$ & $0 \to 0$ & $0 \to 0$ & $0 \to 0$ & $0 \to 0$ \\
 38A & $38{+}$ & $1,2 \to 0$ & $0 \to 0$ & $0 \to 0$ & $0 \to 0$ & $0 \to 0$ \\
 39A & $39{+}$ & $0 \to 0$ & $1 \to 0$ & $0 \to 0$ & $0 \to 0$ & $0 \to 0$ \\
 39B & $39|3{+}$ & $0 \to 0$ & $\infty$ & $0 \to 0$ & $0 \to 0$ & $0 \to 0$ \\
 39CD & $39{+}39$ & $0 \to 0$ & $0 \to 0$ & $0 \to 0$ & $0 \to 0$ & $0 \to 0$ \\
 40A & $40|4{+}$ & $\infty$ & $0,1 \to 0$ & $(1,2,3,5,5)$ & $0 \to 0$ & $0 \to 0$ \\
 40B & $40|2{+}$ & $\infty$ & $0,1 \to 0$ & $0 \to 0$ & $0 \to 0$ & $0 \to 0$ \\
 40CD & $40|2{+}20$ & $\infty$ & $0 \to 0$ & $0 \to 0$ & $0 \to 0$ & $0 \to 0$ \\
 41A & $41{+}$ & $1 \to 0$ & $0 \to 0$ & $0 \to 0$ & $0 \to 0$ & $0 \to 0$ \\
 42A & $42{+}$ & $1 \to 0$ & $1 \to 0$ & $0 \to 0$ & $0 \to 0$ & $0 \to 0$ \\
 42B & $42{+}6,14,21$ & $1 \to 0$ & $0 \to 0$ & $0 \to 0$ & $0 \to 0$ & $0 \to 0$ \\
 42C & $42|3{+}7$ & $0 \to 0$ & $\infty$ & $0 \to 0$ & $0 \to 0$ & $0 \to 0$ \\
 42D & $42{+}3,14,42$ & $0 \to 0$ & $0 \to 0$ & $0 \to 0$ & $0 \to 0$ & $0 \to 0$ \\
 44AB & $44{+}$ & $1 \to 1,0$ & $0 \to 0$ & $0 \to 0$ & $0 \to 0$ & $0 \to 0$ \\
 45A & $45{+}$ & $0 \to 0$ & $1 \to 0$ & $0 \to 0$ & $0 \to 0$ & $0 \to 0$ \\
 46AB & $46{+}23$ & $0 \to 0$ & $0 \to 0$ & $0 \to 0$ & $0 \to 0$ & $0 \to 0$ \\
 46CD & $46{+}$ & $0 \to 0$ & $0 \to 0$ & $0 \to 0$ & $0 \to 0$ & $0 \to 0$ \\
 47AB & $47{+}$ & $0 \to 0$ & $0 \to 0$ & $0 \to 0$ & $0 \to 0$ & $0 \to 0$ \\
 48A & $48|2{+}$ & $\infty$ & $0,1 \to 0$ & $0 \to 0$ & $0 \to 0$ & $0 \to 0$ \\
 50A & $50{+}$ & $0 \to 0$ & $0 \to 0$ & $(1,2,3,5,5)$ & $0 \to 0$ & $0 \to 0$ \\
 51A & $51{+}$ & $1 \to 0$ & $0,1 \to 0$ & $0 \to 0$ & $0 \to 0$ & $0 \to 0$ \\
 52A & $52|2{+}$ & $\infty$ & $0 \to 0$ & $0 \to 0$ & $0 \to 0$ & $0 \to 0$ \\
 52B & $52|2{+}26$ & $\infty$ & $0 \to 0$ & $0 \to 0$ & $0 \to 0$ & $0 \to 0$ \\
 54A & $54{+}$ & $0 \to 0$ & $(1,2,4,8,8)$ & $0 \to 0$ & $0 \to 0$ & $0 \to 0$ \\
 55A & $55{+}$ & $0 \to 0$ & $0 \to 0$ & $0 \to 0$ & $0 \to 0$ & $0 \to 0$ \\
 56A & $56{+}$ & $1 \to 0$ & $0 \to 0$ & $0 \to 0$ & $0 \to 0$ & $0 \to 0$ \\
 56BC & $56|4{+}14$ & $\infty$ & $0 \to 0$ & $0 \to 0$ & $0 \to 0$ & $0 \to 0$ \\
 57A & $57|3{+}$ & $0 \to 0$ & $\infty$ & $0 \to 0$ & $0 \to 0$ & $0 \to 0$ \\
 59AB & $59{+}$ & $0 \to 0$ & $0 \to 0$ & $0 \to 0$ & $0 \to 0$ & $0 \to 0$ \\
 60A & $60|2{+}$ & $\infty$ & $1 \to 0$ & $0 \to 0$ & $0 \to 0$ & $0 \to 0$ \\
 60B & $60{+}$ & $1 \to 0$ & $0 \to 0$ & $0 \to 0$ & $0 \to 0$ & $0 \to 0$ \\
 60C & $60{+}4,15,60$ & $0 \to 0$ & $0 \to 0$ & $0 \to 0$ & $0 \to 0$ & $0 \to 0$ \\
 60D & $60{+}12,15,20$ & $0 \to 0$ & $0 \to 0$ & $0 \to 0$ & $0 \to 0$ & $0 \to 0$ \\
 60E & $60|2{+}5,6,30$ & $\infty$ & $0,1 \to 0$ & $0 \to 0$ & $0 \to 0$ & $0 \to 0$ \\
 60F & $60|6{+}10$ & $\infty$ & $\infty$ & $0 \to 0$ & $0 \to 0$ & $0 \to 0$ \\
 62AB & $62{+}$ & $0 \to 0$ & $0,1 \to 0$ & $0 \to 0$ & $0 \to 0$ & $0 \to 0$ \\
 66A & $66{+}$ & $1 \to 0$ & $0 \to 0$ & $0 \to 0$ & $0 \to 0$ & $0 \to 0$ \\
 66B & $66{+}6,11,66$ & $0 \to 0$ & $0 \to 0$ & $0 \to 0$ & $0 \to 0$ & $0 \to 0$ \\
 68A & $68|2{+}$ & $\infty$ & $0 \to 0$ & $0 \to 0$ & $0 \to 0$ & $0 \to 0$ \\
 69AB & $69{+}$ & $0 \to 0$ & $0 \to 0$ & $0 \to 0$ & $0 \to 0$ & $0 \to 0$ \\
 70A & $70{+}$ & $1 \to 0$ & $0 \to 0$ & $0 \to 0$ & $0 \to 0$ & $0 \to 0$ \\
 70B & $70{+}10,14,35$ & $0 \to 0$ & $0 \to 0$ & $0 \to 0$ & $0 \to 0$ & $0 \to 0$ \\
 71AB & $71{+}$ & $0 \to 0$ & $0 \to 0$ & $0 \to 0$ & $0 \to 0$ & $0 \to 0$ \\
 78A & $78{+}$ & $0 \to 0$ & $0,1 \to 0$ & $0 \to 0$ & $0 \to 0$ & $0 \to 0$ \\
 78BC & $78{+}6,26,39$ & $0 \to 0$ & $0 \to 0$ & $0 \to 0$ & $0 \to 0$ & $0 \to 0$ \\
 84A & $84|2{+}$ & $\infty$ & $0,1 \to 0$ & $0 \to 0$ & $0 \to 0$ & $0 \to 0$ \\
 84B & $84|2{+}6,14,21$ & $\infty$ & $0 \to 0$ & $0 \to 0$ & $0 \to 0$ & $0 \to 0$ \\
 84C & $84|3{+}$ & $0 \to 0$ & $\infty$ & $0 \to 0$ & $0 \to 0$ & $0 \to 0$ \\
 87AB & $87{+}$ & $0 \to 0$ & $0 \to 0$ & $0 \to 0$ & $0 \to 0$ & $0 \to 0$ \\
 88AB & $88|2{+}$ & $\infty$ & $0 \to 0$ & $0 \to 0$ & $0 \to 0$ & $0 \to 0$ \\
 92AB & $92{+}$ & $0 \to 0$ & $0 \to 0$ & $0 \to 0$ & $0 \to 0$ & $0 \to 0$ \\
 93AB & $93|3{+}$ & $0 \to 0$ & $\infty$ & $0 \to 0$ & $0 \to 0$ & $0 \to 0$ \\
 94AB & $94{+}$ & $0 \to 0$ & $0 \to 0$ & $0 \to 0$ & $0 \to 0$ & $0 \to 0$ \\
 95AB & $95{+}$ & $0 \to 0$ & $0 \to 0$ & $0 \to 0$ & $0 \to 0$ & $0 \to 0$ \\
 104AB & $104|4{+}$ & $\infty$ & $0 \to 0$ & $0 \to 0$ & $0 \to 0$ & $0 \to 0$ \\
 105A & $105{+}$ & $0 \to 0$ & $0 \to 0$ & $0 \to 0$ & $0 \to 0$ & $0 \to 0$ \\
 110A & $110{+}$ & $0 \to 0$ & $0 \to 0$ & $0 \to 0$ & $0 \to 0$ & $0 \to 0$ \\
 119AB & $119{+}$ & $0 \to 0$ & $0 \to 0$ & $0 \to 0$ & $0 \to 0$ & $0 \to 0$
\label{tbl:full_annihilation}
\end{longtable}
\end{scriptsize}

\section{Power maps}\label{apndx:web}
For each $p \in \{3, 5, 7, 11\}$ we record here the structure of the collection of groups whose Hauptmoduln are $p$-adically annihilated by $U_p$ (the case $p = 2$ appears as Fig.~\ref{fig:2web} in the introduction). In each diagram below, we write the groups $\G$ such that $\mt_\G$ is $p$-adically annihilated by~$U_p$ but $\mt_\G | U_p \ne 0$, and all of the powers of such groups. Solid lines indicate power maps, groups in white boxes satisfy $\mt_\G | U_p = 0$, and groups in black boxes are not $p$-adically annihilated by $U_p$ at all.

\begin{figure}[h!] \centering
 \begin{tikzpicture}[scale=1.1]
 \inftyzero
 \small{
 \node at ( 0, 0) (1) {$1$};
 \node at (-2, -1) (3) {$3$};
 \node at ( 0, -1) (2+) {$2{+}$};
 \node at ( 2, -1) (3+) {$3{+}$};
 \node[infty] at (-3, -2) (9) {$9$};
 \node at (-2, -2) (9+) {$9{+}$};
 \node at (-1, -2) (6+2) {$6{+}2$};
 \node[zero] at ( 0, -2) (4I2+) {$4|2{+}$};
 \node at ( 1, -2) (6+) {$6{+}$};
 \node at (-3, -3) (27+) {$27{+}$};
 \node[infty] at (-2, -3) (18+2) {$18{+}2$};
 \node at (-1, -3) (18+) {$18{+}$};
 \node[zero] at ( 0, -3) (12I2+2) {$12|2{+}2$};
 \node at ( 1, -3) (8I4+) {$8|4{+}$};
 \node at (-3, -4) (54+) {$54{+}$};
 \node at ( 0, -4) (24I4+2) {$24|4{+}2$};

 \draw (1) -- (3) -- (9) -- (27+) -- (54+);
 \draw (9) -- (18+2) -- (54+);
 \draw (3) -- (9+) -- (18+);
 \draw (3) -- (6+2) -- (18+2);
 \draw (1) -- (2+) -- (6+2) -- (18+);
 \draw (6+2) -- (12I2+2) -- (24I4+2);
 \draw (2+) -- (4I2+) -- (12I2+2);
 \draw (4I2+) -- (8I4+) -- (24I4+2);
 \draw (1) -- (3+) -- (6+) -- (2+);
 }
 \end{tikzpicture}

 \caption{Power maps for $3$-adically annihilated Hauptmoduln.} \label{fig:3web}
\end{figure}

\begin{figure}[h!] \centering
 \begin{tikzpicture}[scale=.94]
 \inftyzero
 \small{
 \node at ( 0, 0) (1) {$1$};
 \node at (-5, -1) (3+) {$3{+}$};
 \node at (-4, -1) (5+) {$5{+}$};
 \node at (-2, -1) (2+) {$2{+}$};
 \node at (-1, -1) (7+) {$7{+}$};
 \node at ( 0, -1) (5) {$5$};
 \node at ( 2, -1) (3I3) {$3|3$};
 \node[zero] at (6, -1) (2) {$2$};
 \node[zero] at (9, -1) (3) {$3$};
 \node at (-6, -2) (6+) {$6{+}$};
 \node at (-5, -2) (15+) {$15{+}$};
 \node at (-4, -2) (10+) {$10{+}$};
 \node at (-3, -2) (4I2+) {$4|2{+}$};
 \node at (-2, -2) (35+) {$35{+}$};
 \node at (-1, -2) (10+2) {$10{+}2$};
 \node at ( 0, -2) (25+) {$25{+}$};
 \node at ( 2, -2) (15I3) {$15|3$};
 \node at ( 1, -2) (21I3+) {$21|3{+}$};
 \node at ( 3, -2) (6I3) {$6|3$};
 \node[zero] at (4.5, -2) (10+10) {$10{+}10$};
 \node[zero] at (6, -2) (4+) {$4{+}$};
 \node[zero] at (7, -2) (4) {$4$};
 \node[zero] at (8, -2) (4I2) {$4|2$};
 \node[zero] at (9, -2) (6+6) {$6{+}6$};
 \node at (-5, -3) (30+) {$30{+}$};
 \node at (-4, -3) (20I2+) {$20|2{+}$};
 \node at (-3, -3) (8I4+) {$8|4{+}$};
 \node at ( 0, -3) (50+) {$50{+}$};
 \node at ( 3, -3) (30I3+10) {$30|3{+}10$};
 \node at ( 5, -3) (12I3+) {$12|3{+}$};
 \node[zero] at (7, -3) (8+) {$8{+}$};
 \node at ( 8, -3) (8I4) {$8|4$};
 \node[zero] at (9, -3) (12I2+6) {$12|2{+}6$};
 \node at (-4, -4) (40I4+) {$40|4{+}$};
 \node at ( 7, -4) (16I2+) {$16|2{+}$};
 \node at ( 8.5, -4.5) (24I4+6) {$24|4{+}6$};

 \draw (1) -- (3+) -- (6+) -- (30+) -- (15+) -- (3+);
 \draw (15+) -- (5+) -- (1) -- (2+) -- (6+);
 \draw (5+) -- (10+) -- (30+);
 \draw (4I2+) -- (2+) -- (10+) -- (20I2+) -- (40I4+) -- (8I4+) -- (4I2+) -- (20I2+);
 \draw (5+) -- (35+) -- (7+) -- (1) -- (5) -- (25+) -- (50+) -- (10+2) -- (2+);
 \draw (10+2) -- (5) -- (15I3) -- (3I3) -- (21I3+) -- (7+);
 \draw (15I3) -- (30I3+10) -- (6I3) -- (3I3) -- (1) -- (2) -- (6I3) -- (12I3+) -- (4+) -- (2) -- (10+10) -- (30I3+10);
 \draw (16I2+) -- (8+) -- (4) -- (2) -- (4I2) -- (8I4) -- (24I4+6) -- (12I2+6) -- (4I2);
 \draw (1) -- (3) -- (6+6) -- (12I2+6);
 \draw (2) -- (6+6);
 }
 \end{tikzpicture}

 \caption{Power maps for $5$-adically annihilated Hauptmoduln.} \label{fig:5web}
\end{figure}

\begin{figure}[h!]\centering
\begin{minipage}[b]{.45\textwidth} \centering
 \begin{tikzpicture}[scale=.75]
 \inftyzero
 \small{
 \node at ( 0, 0) (1) {$1$};
 \node at (-3, -2) (3I3) {$3|3$};
 \node at (-1, -2) (7+) {$7{+}$};
 \node at ( 1, -2) (2+) {$2{+}$};
 \node at ( 3, -2) (7) {$7$};
 \node at (-3, -4) (21I3+) {$21|3{+}$};
 \node at (-1, -4) (14+) {$14{+}$};
 \node at ( 1, -4) (4I2+) {$4|2{+}$};
 \node at (-1, -6) (28I2+) {$28|2{+}$};
 \node at ( 1, -6) (8I4+) {$8|4{+}$};

 \draw (7) -- (1) -- (3I3) -- (21I3+) -- (7+) -- (1) -- (2+) -- (14+) -- (7+);
 \draw (14+) -- (28I2+) -- (4I2+) -- (2+);
 \draw (4I2+) -- (8I4+);
 }
 \end{tikzpicture}

 \caption{Power maps for $7$-adically annihilated Hauptmoduln}
\end{minipage}\quad
\begin{minipage}[b]{.45\textwidth}
 \centering
 \begin{tikzpicture}[scale=1.5]
 \inftyzero
 \small{
 \node at ( 0, 0) (1) {$1$};
 \node at (-1, -1) (3I3) {$3|3$};
 \node[zero] at ( 0, -1) (2+) {$2{+}$};
 \node at ( 1, -1) (11+) {$11{+}$};
 \node at ( 0, -2) (4I2+) {$4|2{+}$};

 \draw (3I3) -- (1) -- (2+) -- (4I2+);
 \draw (1) -- (11+);
 }
 \end{tikzpicture}

 \caption{Power maps for $11$-adically annihilated Hauptmoduln}
\end{minipage}
\label{fig:7-11web}
\end{figure}

\subsection*{Acknowledgements}
We wish to thank Ken Ono, John Duncan, and Larry Rolen for their support and guidance, as well as Frank Calegari for useful information regarding lifts of characteristic $p$ modular forms for the primes $p = 2,3$. We thank our referees for helpful comments. We also thank Emory University, Princeton University, the Asa Griggs Candler Fund, and NSF grant DMS-1557960.

\pdfbookmark[1]{References}{ref}
\LastPageEnding


\begin{thebibliography}{99}
\footnotesize\itemsep=0pt

\bibitem{Andersen13}
Andersen N., Jenkins P., Divisibility properties of coefficients of level~{$p$}
 modular functions for genus zero primes, \href{https://doi.org/10.1090/S0002-9939-2012-11434-0}{\textit{Proc. Amer. Math. Soc.}}
 \textbf{141} (2013), 41--53, \href{https://arxiv.org/abs/1106.1188}{arXiv:1106.1188}.

\bibitem{Atkin67}
Atkin A.O.L., Proof of a conjecture of {R}amanujan, \href{https://doi.org/10.1017/S0017089500000045}{\textit{Glasgow Math.~J.}}
 \textbf{8} (1967), 14--32.

\bibitem{Lehner70}
Atkin A.O.L., Lehner J., Hecke operators on {$\Gamma_{0}(m)$}, \href{https://doi.org/10.1007/BF01359701}{\textit{Math.
 Ann.}} \textbf{185} (1970), 134--160.

\bibitem{Borcherds92}
Borcherds R.E., Monstrous moonshine and monstrous {L}ie superalgebras,
 \href{https://doi.org/10.1007/BF01232032}{\textit{Invent. Math.}} \textbf{109} (1992), 405--444.

\bibitem{Borcherds98}
Borcherds R.E., Modular moonshine.~{III}, \href{https://doi.org/10.1215/S0012-7094-98-09305-X}{\textit{Duke Math.~J.}} \textbf{93}
 (1998), 129--154, \href{https://arxiv.org/abs/math.QA/9801101}{arXiv:math.QA/9801101}.

\bibitem{Borcherds96}
Borcherds R.E., Ryba A.J.E., Modular moonshine.~{II}, \href{https://doi.org/10.1215/S0012-7094-96-08315-5}{\textit{Duke Math.~J.}}
 \textbf{83} (1996), 435--459.

\bibitem{Calegari}
Calegari F., Congruences between modular forms, available at
 \url{http://swc.math.arizona.edu/aws/2013/2013CalegariLectureNotes.pdf}.

\bibitem{Carnahan16}
Carnahan S., Generalized moonshine, {IV}: {M}onstrous {L}ie algebras,
 \href{https://arxiv.org/abs/1208.6254}{arXiv:1208.6254}.

\bibitem{Atlas}
Conway J.H., Curtis R.T., Norton S.P., Parker R.A., Wilson R.A., Atlas of
 finite groups. Maximal subgroups and ordinary characters for simple groups,
 Oxford University Press, Eynsham, 1985.

\bibitem{Conway79}
Conway J.H., Norton S.P., Monstrous moonshine, \href{https://doi.org/10.1112/blms/11.3.308}{\textit{Bull. London Math. Soc.}}
 \textbf{11} (1979), 308--339.

\bibitem{Conway04}
Conway J., McKay J., Sebbar A., On the discrete groups of {M}oonshine,
 \href{https://doi.org/10.1090/S0002-9939-04-07421-0}{\textit{Proc. Amer. Math. Soc.}} \textbf{132} (2004), 2233--2240.

\bibitem{DeHority18}
DeHority S., Gonzalez X., Vafa N., Van~Peski R., Moonshine for all finite
 groups, \href{https://doi.org/10.1007/s40687-018-0133-5}{\textit{Res. Math. Sci.}} \textbf{5} (2018), 14, 34~pages,
 \href{https://arxiv.org/abs/1707.05249}{arXiv:1707.05249}.

\bibitem{Diamond_Shurman}
Diamond F., Shurman J., A first course in modular forms,\textit{Graduate Texts
 in Mathematics}, Vol.~228, \href{https://doi.org/10.1007/978-0-387-27226-9}{Springer-Verlag}, New York, 2005.

\bibitem{Duncan15}
Duncan J.F.R., Griffin M.J., Ono K., Proof of the umbral moonshine conjecture,
 \href{https://doi.org/10.1186/s40687-015-0044-7}{\textit{Res. Math. Sci.}} \textbf{2} (2015), 26, 47~pages,
 \href{https://arxiv.org/abs/1503.01472}{arXiv:1503.01472}.

\bibitem{Duncan15B}
Duncan J.F.R., Mack-Crane S., The moonshine module for {C}onway's group,
 \href{https://doi.org/10.1017/fms.2015.7}{\textit{Forum Math. Sigma}} \textbf{3} (2015), e10, 52~pages,
 \href{https://arxiv.org/abs/1409.3829}{arXiv:1409.3829}.

\bibitem{Elkies05}
Elkies N., Ono K., Yang T., Reduction of {CM} elliptic curves and modular
 function congruences, \href{https://doi.org/10.1155/IMRN.2005.2695}{\textit{Int. Math. Res. Not.}} \textbf{2005} (2005),
 2695--2707, \href{https://arxiv.org/abs/math.NT/0512350}{arXiv:math.NT/0512350}.

\bibitem{Ferenbaugh93}
Ferenbaugh C.R., The genus-zero problem for {$n|h$}-type groups, \href{https://doi.org/10.1215/S0012-7094-93-07202-X}{\textit{Duke
 Math.~J.}} \textbf{72} (1993), 31--63.

\bibitem{Frenkel84}
Frenkel I.B., Lepowsky J., Meurman A., A natural representation of the
 {F}ischer--{G}riess {M}onster with the modular function {$J$} as character,
 \href{https://doi.org/10.1073/pnas.81.10.3256}{\textit{Proc. Nat. Acad. Sci. USA}} \textbf{81} (1984), 3256--3260.

\bibitem{Frenkel85}
Frenkel I.B., Lepowsky J., Meurman A., A moonshine module for the {M}onster, in
 Vertex Operators in Mathematics and Physics ({B}erkeley, {C}alif., 1983),
 \href{https://doi.org/10.1007/978-1-4613-9550-8_12}{\textit{Math. Sci. Res. Inst. Publ.}}, Vol.~3, Springer, New York, 1985 231--273.

\bibitem{Gannon16}
Gannon T., Much ado about {M}athieu, \href{https://doi.org/10.1016/j.aim.2016.06.014}{\textit{Adv. Math.}} \textbf{301} (2016),
 322--358, \href{https://arxiv.org/abs/1211.5531}{arXiv:1211.5531}.

\bibitem{GAP}
GAP -- Groups, Algorithms, and Programming, Version 4.9.2, 2018,
 \url{https://www.gap-system.org}.

\bibitem{Gouvea}
Gouv\^{e}a F.Q., Arithmetic of {$p$}-adic modular forms, \textit{Lecture Notes
 in Mathematics}, Vol.~1304, \href{https://doi.org/10.1007/BFb0082111}{Springer-Verlag}, Berlin, 1988.

\bibitem{Harada96}
Harada K., Lang M.L., The {M}c{K}ay--{T}hompson series associated with the
 irreducible characters of the {M}onster, in Moonshine, the {M}onster, and
 Related Topics ({S}outh {H}adley, {MA}, 1994), \textit{Contemp. Math.}, Vol.~193, \href{https://doi.org/10.1090/conm/193/02367}{Amer. Math. Soc.}, Providence, RI, 1996, 93--111, \href{https://arxiv.org/abs/q-alg/9412013}{arXiv:q-alg/9412013}.

\bibitem{Harvey15}
Harvey J.A., Rayhaun B.C., Traces of singular moduli and moonshine for the
 {T}hompson group, \href{https://doi.org/10.4310/CNTP.2016.v10.n1.a2}{\textit{Commun. Number Theory Phys.}} \textbf{10} (2016),
 23--62, \href{https://arxiv.org/abs/1504.08179}{arXiv:1504.08179}.

\bibitem{Hida93}
Hida H., Elementary theory of {$L$}-functions and {E}isenstein series,
 \textit{London Mathematical Society Student Texts}, Vol.~26, \href{https://doi.org/10.1017/CBO9780511623691}{Cambridge
 University Press}, Cambridge, 1993.

\bibitem{Jenkins15}
Jenkins P., Thornton D.J., Congruences for coefficients of modular functions,
 \href{https://doi.org/10.1007/s11139-014-9628-x}{\textit{Ramanujan~J.}} \textbf{38} (2015), 619--628, \href{https://arxiv.org/abs/1404.0699}{arXiv:1404.0699}.

\bibitem{Jochnowitz82}
Jochnowitz N., Congruences between systems of eigenvalues of modular forms,
 \href{https://doi.org/10.2307/1999772}{\textit{Trans. Amer. Math. Soc.}} \textbf{270} (1982), 269--285.

\bibitem{Katz73}
Katz N.M., {$p$}-adic properties of modular schemes and modular forms, in
 Modular Functions of One Variable,~{III} ({P}roc. {I}nternat. {S}ummer
 {S}chool, {U}niv. {A}ntwerp, {A}ntwerp, 1972), \textit{Lecture Notes in
 Math.}, Vol.~350, Editors W.~Kuijk, J.-P.~Serre, \href{https://doi.org/10.1007/978-3-540-37802-0_3}{Springer}, Berlin, 1973,
 69--190.

\bibitem{Katz76}
Katz N.M., A result on modular forms in characteristic~{$p$}, in Modular
 functions of one variable, {V} ({P}roc. {S}econd {I}nternat. {C}onf., {U}niv.
 {B}onn, {B}onn, 1976), \textit{Lecture Notes in Math.}, Vol.~601, Editors
 J.-P.~Serre, D.B.~Zagier, \href{https://doi.org/10.1007/BFb0063944}{Springer}, Berlin, 1977, 53--61.

\bibitem{Larson16}
Larson H., Coefficients of {M}c{K}ay--{T}hompson series and distributions of
 the moonshine module, \href{https://doi.org/10.1090/proc/13228}{\textit{Proc. Amer. Math. Soc.}} \textbf{144} (2016),
 4183--4197, \href{https://arxiv.org/abs/1508.03742}{arXiv:1508.03742}.

\bibitem{Lehner49a}
Lehner J., Divisibility properties of the {F}ourier coefficients of the modular
 invariant~{$j(\tau)$}, \href{https://doi.org/10.2307/2372101}{\textit{Amer.~J. Math.}} \textbf{71} (1949), 136--148.

\bibitem{Lehner49b}
Lehner J., Further congruence properties of the {F}ourier coefficients of the
 modular invariant~{$j(\tau)$}, \href{https://doi.org/10.2307/2372252}{\textit{Amer.~J. Math.}} \textbf{71} (1949),
 373--386.

\bibitem{Norton87}
Norton S.P., Generalized moonshine, in The {A}rcata {C}onference on
 {R}epresentations of {F}inite {G}roups ({A}rcata, {C}alif., 1986),
 \textit{Proc. Sympos. Pure Math.}, Vol.~47, Amer. Math. Soc., Providence, RI,
 1987, 208--210.

\bibitem{Pari/GP}
Pari/{GP} ({V}ersion 2.11.0), University of Bordeaux, 2018,
 \url{http://pari.math.u-bordeaux.fr}.

\bibitem{Ryba96}
Ryba A.J.E., Modular {M}oonshine?, in Moonshine, the {M}onster, and Related
 Topics ({S}outh {H}adley, {MA}, 1994), \textit{Contemp. Math.}, Vol.~193,
 \href{https://doi.org/10.1090/conm/193/02378}{Amer. Math. Soc.}, Providence, RI, 1996, 307--336.

\bibitem{Sage}
{S}ageMath, the {S}age {M}athematics {S}oftware {S}ystem ({V}ersion 8.3), 2018,
 \url{http://www.sagemath.org}.

\bibitem{Serre72}
Serre J.-P., Formes modulaires et fonctions z\^{e}ta {$p$}-adiques, in Modular
 functions of one variable,~{III} ({P}roc. {I}nternat. {S}ummer {S}chool,
 {U}niv. {A}ntwerp, 1972), \textit{Lecture Notes in Math.}, Vol.~350, Editors
 W.~Kuijk, \mbox{J.-P.~Serre}, \href{https://doi.org/10.1007/978-3-540-37802-0_4}{Springer}, Berlin, 1973, 191--268.

\bibitem{Serre76}
Serre J.-P., Divisibilit\'{e} de certaines fonctions arithm\'{e}tiques,
 \textit{Enseignement Math.} \textbf{22} (1976), 227--260.

\bibitem{Sturm87}
Sturm J., On the congruence of modular forms, in Number Theory ({N}ew {Y}ork,
 1984--1985), \textit{Lecture Notes in Math.}, Vol. 1240, \href{https://doi.org/10.1007/BFb0072985}{Springer}, Berlin,
 1987, 275--280.

\bibitem{Thompson79b}
Thompson J.G., Finite groups and modular functions, \href{https://doi.org/10.1112/blms/11.3.347}{\textit{Bull. London Math.
 Soc.}} \textbf{11} (1979), 347--351.

\bibitem{Thompson79}
Thompson J.G., Some numerology between the {F}ischer--{G}riess {M}onster and the
 elliptic modular function, \href{https://doi.org/10.1112/blms/11.3.352}{\textit{Bull. London Math. Soc.}} \textbf{11}
 (1979), 352--353.

\bibitem{Wilson88}
Wilson R.A., The odd-local subgroups of the {M}onster, \href{https://doi.org/10.1017/S1446788700031323}{\textit{J.~Austral.
 Math. Soc. Ser.~A}} \textbf{44} (1988), 1--16.

\end{thebibliography}
\end{document}